\documentclass[11pt,a4paper]{article}
\usepackage[margin=0.75in]{geometry}
\usepackage[utf8]{inputenc}
\usepackage{cmap}
\usepackage[T1]{fontenc}
\usepackage{lmodern}
\usepackage[square, numbers]{natbib}     %[square, numbers]
\usepackage[colorlinks=true]{hyperref}
\usepackage{amssymb, amsthm, amsmath, amsfonts}
\usepackage{authblk, dsfont, mathrsfs, color, bm, enumitem, mathtools}
\usepackage{subfigure, graphicx, transparent}
\usepackage{dirtytalk}
\graphicspath{{figures/}}
\usepackage[dvipsnames]{xcolor}
\usepackage{comment}
\hypersetup{colorlinks, linkcolor={red!80!black}, citecolor={blue!80!black}, urlcolor={ForestGreen}}
\theoremstyle{definition}
\newtheorem{theorem}{Theorem}[section]
\newtheorem{definition}[theorem]{Definition}
\newtheorem{assumption}[theorem]{Assumption}
\newtheorem{proposition}[theorem]{Proposition}
\newtheorem{corollary}[theorem]{Corollary}
\newtheorem{lemma}[theorem]{Lemma}
\newtheorem{remark}[theorem]{Remark}

\DeclareUnicodeCharacter{211D}{\ensuremath{\mathbb R}}

\DeclarePairedDelimiterX{\inner}[2]{\langle}{\rangle}{#1, #2}

\title{Stopping Rules for Stochastic Gradient Descent via Anytime-Valid Confidence Sequences}

\author{Liviu Aolaritei$^\dagger$}
\author{Michael I. Jordan$^{\dagger \ddagger}$}

\affil{$^\dagger$University of California, Berkeley, USA \\
$^\ddagger$Inria Paris, France \protect\\ \texttt{liviu.aolaritei@berkeley.edu, jordan@cs.berkeley.edu}}

\date{}

\begin{document}
\maketitle

\begin{abstract}
The problem of stopping stochastic gradient descent (SGD) in an online manner, based solely on the observed trajectory, is a challenging theoretical problem with significant consequences for applications. While SGD is routinely monitored as it runs, the classical theory of SGD provides guarantees only at pre-specified iteration horizons and offers no valid way to decide, based on the observed trajectory, when further computation is justified. We address this longstanding gap by developing anytime-valid confidence sequences for stochastic gradient methods, which remain valid under continuous monitoring and directly induce statistically valid, trajectory-dependent stopping rules: stop as soon as the current upper confidence bound on an appropriate performance measure falls below a user-specified tolerance. The confidence sequences are constructed using nonnegative supermartingales, are time-uniform, and depend only on observable quantities along the SGD trajectory, without requiring prior knowledge of the optimization horizon. In convex optimization, this yields anytime-valid certificates for weighted suboptimality of projected SGD under general stepsize schedules, without assuming smoothness or strong convexity. In nonconvex optimization, it yields time-uniform certificates for weighted first-order stationarity under smoothness assumptions. We further characterize the stopping-time complexity of the resulting stopping rules under standard stepsize schedules. To the best of our knowledge, this is the first framework that provides statistically valid, time-uniform stopping rules for SGD across both convex and nonconvex settings based solely on its observed trajectory.
\end{abstract}

%--------------------------------------------------------------------------------
%--------------------------------------------------------------------------------
%--------------------------------------------------------------------------------

\section{Introduction}
\label{sec:intro}

Deciding when to stop an optimization algorithm is often treated as a routine design choice. One selects a number of iterations, runs the algorithm, and accepts the resulting iterate. Yet this convention conceals a basic tension. The chosen horizon is necessarily arbitrary: had the algorithm converged more quickly, computation could have been saved; had progress been slower, additional iterations might have led to a meaningfully better solution. Standard optimization theory offers little guidance for resolving this tension while the algorithm is actually running.

This issue is particularly acute for stochastic gradient descent (SGD), which in practice is almost always monitored as it unfolds. Existing analyses of SGD characterize its behavior by answering a forward-looking question: if the algorithm is run for a prescribed number of iterations, what guarantees can be made about optimality or stationarity? In practice, however, users face the inverse problem. At a given iteration, after observing the realized trajectory of the algorithm, one would like to know whether continuing to iterate is justified statistically. This inversion of perspective lies at the heart of practical optimization, where computational budgets are finite and decisions must be made online.

A natural response is to monitor the algorithm’s progress as it unfolds and to stop once some criterion appears satisfactory. However, this adaptive monitoring fundamentally alters the statistical setting. Guarantees derived for fixed iteration horizons do not remain valid when the algorithm is inspected repeatedly and stopping decisions are based on the observed trajectory. Classical confidence bounds implicitly assume that the stopping rule is fixed in advance; once this assumption is violated, repeated checking induces multiple testing effects that invalidate standard statistical conclusions. As a result, widely used stopping heuristics, including fixed budgets, visually detected plateaus, or ad hoc thresholds, provide no quantifiable protection against premature or overly conservative termination.

The consequences are not merely theoretical. In large-scale learning, reinforcement learning, and scientific computing, the cost of unnecessary computation can be substantial, while stopping too early may leave significant performance gains unrealized. What is missing is a mechanism that translates the observed trajectory of SGD into statistically reliable, real-time evidence about progress, without requiring the stopping time to be specified in advance. This raises the central question of this paper:

\begin{center}
\emph{How close is the current iterate to optimality (or stationarity)?}
\end{center}

In this work, we address this gap by formulating stopping for SGD as a problem of sequential statistical inference. Our approach is based on anytime-valid confidence sequences, which are designed to remain valid under continuous monitoring and arbitrary data-dependent stopping. By adapting this framework to the dynamics of stochastic gradient methods, we obtain time-uniform, data-dependent certificates of performance that can be inspected at every iteration without invalidating their guarantees. These certificates directly give rise to simple, statistically valid stopping rules that depend only on the observed trajectory of the algorithm.

This framework applies to both convex and nonconvex optimization. In convex settings, it yields anytime-valid certificates for suboptimality that allow one to stop SGD with a guaranteed level of optimality. In nonconvex settings, where global optimality is unattainable in general, it yields time-uniform certificates for first-order stationarity. In both cases, the resulting stopping rules are compatible with standard stepsize schedules and preserve their validity under arbitrary adaptive monitoring.

%--------------------------------------------------------------------------------

\subsection{Problem formulation}
\label{subsec:problem:form}

We consider stochastic optimization problems of the form
\[
    \min_{x \in \mathcal{X}} f(x),
\]
where $\mathcal{X} \subseteq \mathbb{R}^d$ is a closed convex set and $f:\mathcal{X} \to \mathbb{R}$ is a (possibly nonconvex) objective function. Throughout, we study stochastic gradient descent with projection onto $\mathcal{X}$, defined recursively by
\begin{equation}
\label{eq:proj-sgd-update}
    x_{t+1} = \Pi_{\mathcal{X}}(x_t - \eta_t g_t)
    \qquad \forall t \ge 1,
\end{equation}
where $\Pi_{\mathcal{X}}$ denotes Euclidean projection onto $\mathcal{X}$, $\{g_t\}_{t\ge1}$ are stochastic gradients, and $\{\eta_t\}_{t\ge1}$ is a stepsize sequence. We define the natural filtration $\mathcal{F}_t := \sigma(x_1, g_1, \ldots, x_t, g_t)$, for $t \ge 1$, and assume that the stepsizes are predictable with respect to this filtration.

\begin{assumption}[Stepsizes]
\label{as:stepsizes}
The stepsizes $\{\eta_t\}_{t\ge1}$ are $\mathcal{F}_{t-1}$-measurable for all $t \ge 1$.
\end{assumption}

Assumption~\ref{as:stepsizes} allows for arbitrary data-dependent stepsize rules, provided they depend only on the past trajectory. In particular, it covers standard constant stepsizes, diminishing schedules such as $\eta_t=\eta_0\, t^{-1/2}$, classical stochastic approximation stepsizes satisfying $\sum_{t=1}^\infty \eta_t = \infty$ and $\sum_{t=1}^\infty \eta_t^2 < \infty$, as well as adaptive stepsizes that are functions of the observed iterates and gradients (e.g., AdaGrad-type schedules and related adaptive methods).

The stochastic gradients are assumed to satisfy the following conditions, which are common to both the convex and nonconvex settings.

\begin{assumption}[Stochastic gradients I]
\label{as:gradients}
For all $t\ge 1$, the stochastic gradient $g_t$ satisfies:
\begin{itemize}
    \item[(i)] \emph{Unbiasedness:} $\mathbb{E}[g_t \mid \mathcal{F}_{t-1}] = \nabla f(x_t)$, and

    \item[(ii)] \emph{Conditional sub-Gaussian noise:} there exists a constant $\sigma^2 > 0$ such that for $\xi_t \coloneqq g_t - \nabla f(x_t)$,
\begin{equation}
\label{eq:cond-subg-noise}
    \mathbb{E}\!\left[
        \exp\!\big(\lambda \langle u, \xi_t \rangle\big)
        \,\Big|\, \mathcal{F}_{t-1}
    \right]
    \le
    \exp\!\left(
        \frac{\lambda^2 \sigma^2 \|u\|^2}{2}
    \right)
    \qquad \text{almost surely }\forall u \in \mathbb{R}^d, \forall\lambda \in \mathbb{R}.
\end{equation}
\end{itemize}
\end{assumption}

Assumption~\ref{as:gradients}(i) is the standard unbiasedness condition used in stochastic approximation. %Assumption~\ref{as:gradients}(ii) ensures that the quadratic term $\|g_t\|^2$ remains controlled, which is classical in nonasymptotic analyses of SGD. 
Moreover, Assumption~\ref{as:gradients}(ii) postulates that the gradient noise is conditionally sub-Gaussian given the past. This condition provides quantitative tail control on the one-step fluctuations of the stochastic gradients and is natural in time-uniform concentration arguments; it is satisfied by bounded, Gaussian, and many light-tailed noise models encountered in practice. In particular, if $\|g_t\|\le G$ almost surely, then \eqref{eq:cond-subg-noise} holds with $\sigma^2 := 4G^2$.

For convenience, we introduce the cumulative quantities
\[
    S_t := \sum_{s=1}^t \eta_s,
    \qquad
    V_t := \sum_{s=1}^t \eta_s^2,
\]
which will appear repeatedly in the analysis. We now distinguish between the convex and nonconvex settings, which differ only in the choice of performance measure and in the additional assumptions imposed on the objective in the nonconvex case.

\bigskip
\noindent\textbf{Convex case.}
In the convex setting, we assume that $f$ is convex and that the set $\mathcal{X}$ is \emph{compact}. This assumption is used only to guarantee a uniform bound on the distance $\|x_t - x^\star\|$ along the SGD trajectory; see Remark~\ref{rmk:compactness} for a discussion of relaxed conditions. Let
\[
    x^\star \in \arg\min_{x \in \mathcal{X}} f(x)
\]
denote an arbitrary minimizer. Since all iterates satisfy $x_t \in \mathcal{X}$, they remain in a bounded neighborhood of $x^\star$:
\[
    \|x_t - x^\star\| \le R_x := \operatorname{diam}(\mathcal{X}) < \infty .
\]
To measure optimization progress, we focus on the weighted average suboptimality
\begin{equation}
\label{eq:weighted:subopt}
    \bar F_t
    :=
    \frac{1}{S_t}
    \sum_{s=1}^t \eta_s \bigl(f(x_s) - f(x^\star)\bigr).
\end{equation}
Defining the corresponding weighted average iterate
\[
    \bar x_t := \frac{1}{S_t} \sum_{s=1}^t \eta_s x_s,
\]
convexity of $f$ implies
\[
    f(\bar x_t) - f(x^\star) \le \bar F_t.
\]
Thus $\bar F_t$ simultaneously controls the performance of the averaged iterate and serves as the natural target for our anytime-valid analysis in the convex case.

\bigskip
\noindent\textbf{Nonconvex case.}
In the nonconvex setting, we take $\mathcal{X} = \mathbb{R}^d$ and impose no convexity or coercivity assumptions on the objective. We assume that $f : \mathbb{R}^d \to \mathbb{R}_+$ is differentiable. The nonnegativity assumption is made without loss of generality by shifting the objective if necessary. We do, however, require a standard smoothness condition on the objective.

\begin{assumption}[Smoothness]
\label{as:smoothness}
The function $f$ is $L$-smooth, that is,
\[
    \|\nabla f(x) - \nabla f(y)\| \le L \|x - y\|,
    \qquad \forall x,y \in \mathbb{R}^d .
\]
\end{assumption}

In the absence of convexity, optimality gaps are no longer an appropriate performance measure. Instead, we quantify progress through a weighted stationarity measure,
\begin{equation}
\label{eq:def-Gbar}
    \bar G_t
    :=
    \frac{1}{S_t}
    \sum_{s=1}^t \eta_s \|\nabla f(x_s)\|^2 .
\end{equation}
This quantity is standard in nonconvex stochastic optimization and captures the extent to which the algorithm has approached a stationary regime. Moreover, since the stepsizes are nonnegative, $\bar G_t$ upper bounds the smallest squared gradient norm encountered along the trajectory:
\[
    \min_{1 \le s \le t} \|\nabla f(x_s)\|^2 \le \bar G_t .
\]
We note that $\bar G_t$ is an \emph{operational} stationarity metric: while small gradient norms certify approximate first-order stationarity, they do not in general provide a calibrated notion of distance to a particular stationary point without additional structure. Our goal in the nonconvex setting is to construct time-uniform, data-dependent upper confidence bounds on $\bar G_t$ that remain valid under continuous monitoring and arbitrary data-dependent stopping.

%--------------------------------------------------------------------------------

\subsection{Contributions}
\label{subsec:contributions}

Our goal is to provide \emph{anytime-valid, data-dependent guarantees} for the progress of stochastic gradient descent across both convex and nonconvex optimization settings. Specifically, given the stochastic gradients $\{g_t\}$, the predictable stepsizes $\{\eta_t\}$, and the observed iterates $\{x_t\}$, we construct time-uniform upper confidence bounds for a suitable performance measure associated with the SGD trajectory. In the convex case, this measure is the weighted average suboptimality $\bar F_t$ defined in \eqref{eq:weighted:subopt}, while in the nonconvex case it is the weighted average stationarity measure $\bar G_t$ defined in \eqref{eq:def-Gbar}. In both settings, the resulting bounds remain valid simultaneously for all times $t \ge 1$, including under continuous monitoring and arbitrary data-dependent stopping. To state our contributions precisely, we recall the formal notion of an anytime-valid upper confidence sequence.

\begin{definition}[Anytime-valid upper confidence sequence]
\label{def:cs}
Let $\{\bar H_t\}_{t\ge 1}$ denote a nonnegative performance process, such as $\bar F_t$ in the convex case or $\bar G_t$ in the nonconvex case. A sequence of random variables $\{U_t(\alpha)\}_{t\ge 1}$ is an \emph{anytime-valid upper confidence sequence} for $\{\bar H_t\}$ if
\[
    \mathbb{P}\!\left(
        \forall t\ge 1:\ \bar H_t \le U_t(\alpha)
    \right)
    \ge 1-\alpha.
\]
\end{definition}

Such time-uniform bounds naturally induce statistically valid stopping rules. Given a target accuracy $\varepsilon>0$ and confidence level $\alpha\in(0,1)$, we define the stopping time
\[
    \tau_\varepsilon :=
    \inf\left\{ t\ge 1 : U_t(\alpha) \le \varepsilon \right\},
\]
which returns the first iterate that is \emph{certified} to have reached the desired performance level. In the convex case, this corresponds to $\varepsilon$-optimality, while in the nonconvex case it corresponds to $\varepsilon$-stationarity (up to the appropriate scaling of the performance measure).

\bigskip
\noindent\textbf{Main contributions.}
Our main contributions can be summarized as follows.

\begin{itemize}
\item[(i)] \textbf{Anytime-valid confidence sequences.}
Under the assumptions of Section~\ref{subsec:problem:form}, we construct anytime-valid upper confidence sequences $\{U_t(\alpha)\}_{t\ge 1}$, in the sense of Definition~\ref{def:cs}, for the weighted average suboptimality $\bar F_t$ in the convex setting and for the weighted average stationarity measure $\bar G_t$ in the nonconvex setting. The resulting bounds are trajectory-adaptive, depend only on quantities generated along the SGD path, and remain valid uniformly over time, including under continuous monitoring and data-dependent choices of stepsizes. In the convex case, compactness of the feasible set yields a fully observable confidence sequence depending only on realized stochastic gradient norms, the stepsizes, and the domain diameter. In the nonconvex case, the confidence sequence controls $\bar G_t$ and becomes fully observable under additional boundedness assumptions on $\|\nabla f(x_t)\|$. Finally, we characterize the rates of decay of these anytime-valid certificates under common stepsize choices, including (square-summable) stochastic approximation stepsizes and the canonical $\eta_t=\eta_0\, t^{-1/2}$ schedule, obtaining explicit bounds in expectation and almost surely along the realized SGD trajectory.

\item[(ii)] \textbf{Statistically valid stopping rules.}
The anytime-valid confidence sequences in (i) directly induce trajectory-dependent stopping rules that certify $\varepsilon$-optimality in the convex case and $\varepsilon$-stationarity in the nonconvex case with confidence $1-\alpha$, without requiring a predetermined optimization horizon. These guarantees hold for \emph{arbitrary predictable stepsize sequences}, including adaptive stepsizes that are functions of the observed iterates and stochastic gradients. To provide concrete insight into the resulting stopping behavior, we analyze the induced certified stopping rules by deriving expected stopping-time bounds (i.e., bounds on $\mathbb{E}[\tau_\varepsilon]$) under commonly used stepsize schedules, including $\eta_t=\eta_0 t^{-1/2}$ and more general stochastic approximation–type schedules.
\end{itemize}

%--------------------------------------------------------------------------------

\subsection{Related work}

Conceptually, this work sits at the intersection of two lines of research. The first concerns \emph{time-uniform statistical inference}, which develops hypothesis tests and confidence bounds that remain valid under continuous monitoring and arbitrary data-dependent stopping. The second concerns \emph{stochastic approximation}, which provides convergence guarantees for stochastic gradient methods in both convex and nonconvex optimization. The present paper brings these perspectives together by adapting nonnegative supermartingale techniques from time-uniform inference to the dynamics of stochastic gradient descent. Below, we position our contributions within the broader literature on time-uniform statistical inference and stochastic optimization.

\medskip

\noindent\textbf{Time-uniform inference and confidence sequences.}
Time-uniform (also called anytime-valid) inference provides uncertainty bounds that remain valid simultaneously over all times and are therefore robust to continuous monitoring and optional stopping. A foundational tool is Ville’s inequality for nonnegative supermartingales~\cite{ville1939}, which converts supermartingale constructions into time-uniform deviation guarantees. Early statistical instances of this idea include classical confidence sequences for basic parameters~\cite{darling1967confidence}, as well as boundary-crossing results for martingales~\cite{robbins1970boundary}. Modern developments systematize these principles through nonnegative supermartingale constructions and, in particular, mixture and stitching arguments, yielding sharp time-uniform concentration inequalities that adapt to variance-type processes~\cite{howard2021time}; see also related time-uniform Chernoff bounds based on exponential supermartingales~\cite{howard2020time} and the broader literature on self-normalized processes~\cite{de2009self}. More recently, these ideas have been reframed through e-values and game-theoretic statistics, providing a unified perspective on optional-stopping-safe inference and a comprehensive entry point to the literature~\cite{ramdas2023game,ramdas2025hypothesis}. While this framework has been developed primarily for sequential testing and estimation, it has begun to influence learning and decision-making problems, including anytime-valid inference for bandits and betting-based estimation~\cite{waudby2024anytime,waudby2024estimating}. Complementing time-uniform validity, expected stopping-time guarantees have also been studied in several strands of the sequential-testing literature. Classical results give bounds on the expected stopping time of the sequential probability ratio test \cite{wald1992sequential}. More recent work studies expected rejection times for bounded-mean hypothesis testing \cite{chugg2023auditing,chen2025optimistic}, as well as for two-sample and independence testing \cite{shekhar2023nonparametric}. Expected rejection-time guarantees have also been developed for sequential hypothesis testing over general classes of composite testing problems \cite{waudby2025universal}. Finally, expected stopping-time guarantees are central in best-arm identification in multi-armed bandits, where sequential tests aim to identify an optimal arm as quickly as possible \cite{garivier2016optimal,kaufmann2016complexity,kaufmann2021mixture,agrawal2020optimal,agrawal2021optimal}. Our work builds on this time-uniform machinery in an optimization context, using supermartingale mixtures to build SGD performance certificates that are valid under arbitrary, data-dependent stopping and that can also be paired with expected stopping-time analyses under standard stepsize regimes.

\medskip

\noindent\textbf{Stochastic approximation and asymptotic normality.}
Stochastic approximation provides the theoretical foundation for stochastic gradient methods. The seminal paper~\cite{robbins1951stochastic} established almost sure convergence of iterative schemes under diminishing stepsizes, while~\cite{kushner2003stochastic} developed a comprehensive asymptotic theory based on ordinary differential equations, covering stability, asymptotic normality, and efficiency. A key advance showed that appropriately averaged iterates achieve asymptotically optimal covariance, providing a principled justification for studying \emph{weighted averages} (which we adopt in this paper) of SGD trajectories~\cite{polyak1992acceleration}. This asymptotic  perspective was sharpened in~\cite{duchi2021asymptotic}, which characterized minimax and asymptotically efficient rates for stochastic optimization, including constrained formulations, which was further extended to nonsmooth stochastic approximation in~\cite{davis2024asymptotic}. Together, these results motivate the use of averaged iterates and enable asymptotic confidence regions for optimal solutions based on limiting normality. However, such guarantees are inherently asymptotic: they do not yield finite-time confidence sets, nor do they provide time-uniform validity under continuous monitoring or adaptive stopping.

\medskip

\noindent\textbf{Fixed-horizon optimality guarantees for convex SGD.}
A vast literature develops \emph{nonasymptotic} convergence guarantees for stochastic gradient (and mirror-descent–type) methods that hold with high probability at a \emph{pre-specified} iteration horizon. Early robust frameworks for stochastic convex optimization established high-probability bounds under bounded or sub-Gaussian noise assumptions, providing optimal oracle complexity guarantees for stochastic approximation and stochastic programming~\cite{nemirovski2009robust,lan2012optimal}. Related work also studied statistical validation of mirror-descent–type stochastic approximation methods, providing confidence bounds for the optimal objective value at a fixed iteration horizon rather than under adaptive or continuously monitored stopping rules~\cite{lan2012validation}. Parallel developments in machine learning emphasized refined nonasymptotic analyses of stochastic approximation algorithms, including explicit bias--variance tradeoffs and finite-time rates for averaged iterates~\cite{moulines2011non,bach2013non,dieuleveut2016nonparametric}. Subsequent work sharpened fixed-horizon high-probability bounds across a range of regimes, including nonsmooth and strongly convex objectives, and clarified when optimal rates can be achieved by averaging versus last-iterate solutions~\cite{rakhlin2012making,shamir2013stochastic,harvey2019tight,harvey2019simple,jain2021making}. Complementary information-theoretic lower bounds characterize the fundamental limits of stochastic convex optimization at fixed confidence levels~\cite{agarwal2012information}. More recent contributions extend fixed-horizon high-probability guarantees beyond light-tailed noise models, including heavy-tailed settings handled via gradient clipping and related robustification techniques~\cite{gorbunov2020stochastic}, as well as general boosting and robust-estimation wrappers that convert low-probability guarantees into high-confidence ones~\cite{davis2021low}. Together, these works provide sharp finite-time guarantees at prescribed horizons and confidence levels. However, these guarantees are intrinsically fixed-time: they are stated for pre-specified horizons and do not remain valid under continuous monitoring or data-dependent stopping without additional union bounds or restarts.

\medskip

\noindent\textbf{Fixed-horizon stationarity guarantees for nonconvex SGD.}
Theoretical guarantees for stochastic gradient methods in \emph{nonconvex} optimization have been studied extensively, with performance typically quantified through first-order stationarity criteria, most commonly bounds on the \emph{squared} gradient norm $\|\nabla f(x)\|^2$ evaluated at a prescribed iteration horizon via a suitable output rule. Foundational stochastic approximation theory establishes almost sure convergence of iterates to stationary sets under diminishing stepsizes and appropriate noise conditions, providing the conceptual basis for modern analyses in nonconvex stochastic optimization~\cite{robbins1951stochastic,kushner2003stochastic,benaim2006dynamics,borkar2008stochastic}. For smooth nonconvex objectives, a widely adopted nonasymptotic benchmark is an $O(t^{-1/2})$ rate for approximate stationarity under unbiased stochastic gradients and bounded-variance or light-tailed noise. Such guarantees are typically stated in expectation for a prescribed output iterate, such as a randomly selected or averaged iterate among the first $t$ iterates, or equivalently in terms of the minimum stationarity measure attained up to time $t$. This benchmark was crystallized and systematically analyzed in influential works~\cite{ghadimi2013stochastic,ghadimi2016accelerated}, and has since served as a reference point for a broad range of subsequent contributions; see also modern syntheses such as~\cite{bottou2018optimization,lan2020first}. Beyond vanilla SGD, extensive work studies refined oracle-complexity guarantees for nonconvex problems, including variance-reduced and recursive-gradient methods for finite-sum or stochastic objectives~\cite{reddi2016stochastic,lei2017less,nguyen2017sarah,fang2018spider,li2021page}, algorithms designed to escape saddle points~\cite{ge2015escaping,jin2017escape,allen2018natasha}, and adaptive or momentum-based schemes in nonconvex landscapes~\cite{zaheer2018adaptive,ward2020adagrad}. Complementary results establish nonasymptotic stationarity guarantees in nonconvex optimization through different mechanisms, including high-probability bounds obtained via gradient clipping or other robustification techniques under heavy-tailed noise~\cite{cutkosky2021high,li2022high}, adaptive first-order methods with refined concentration analyses~\cite{kavis2022high}, and expectation-based analyses of vanilla SGD under heavy-tailed stochastic gradients~\cite{fatkhullin2025can}. While this body of work provides sharp convergence and oracle-complexity guarantees at fixed, pre-specified horizons, it does not address the statistical validity of adaptive stopping decisions based on continuously monitored optimization trajectories.

\medskip

\noindent\textbf{Closest related work.}
The works most closely related to ours study anytime or time-uniform guarantees for stochastic approximation, but differ substantially in assumptions and in the form of the resulting guarantees. Classical results such as \cite{rakhlin2012making} establish high-probability convergence bounds for SGD under strong convexity and bounded stochastic gradients, with guarantees stated uniformly over iterations up to a pre-specified horizon. As such, they do not support data-dependent stopping times. More recent work \cite{feng2023anytime} derives anytime high-probability bounds for stochastic gradient methods with momentum under convexity and smoothness assumptions, using algorithm-specific dynamics and stepsize choices motivated by continuous-time ODE analysis. However, these results are not formulated in terms of trajectory-observable confidence sequences and are not used to derive certified stopping rules. Finally, \cite{yu2025root} studies one-dimensional stochastic root finding and proposes a bisection-based procedure that yields time-uniform confidence intervals for the unknown root, but their guarantees concern inference on the unknown root rather than certificates of optimization suboptimality for projected SGD. In contrast to these works, our results apply to plain projected SGD under general convexity, yield fully explicit and trajectory-dependent confidence sequences for suboptimality, and lead to provably correct stopping rules that remain valid under arbitrary adaptive monitoring.

%--------------------------------------------------------------------------------

\subsection{Organization and notation}

\noindent\textbf{Organization.}
Section~\ref{sec:anytime:CS} develops anytime-valid upper confidence sequences for the weighted average suboptimality in the convex case and for a weighted stationarity measure in the nonconvex case, including fully observable variants, and studies their decay rates under common stepsize schedules. Section~\ref{sec:stopping} uses these confidence sequences to construct statistically valid, trajectory-dependent stopping rules and analyzes the resulting stopping-time complexity under common stepsize schedules. Section~\ref{sec:proofs} contains the proofs of the main results. Finally, the Appendix collects additional proofs and supporting technical lemmas.

\medskip

\noindent\textbf{Notation.}
For a vector $x\in\mathbb{R}^d$, $\|x\|$ denotes the Euclidean norm and $\langle x,y\rangle$ the standard inner product. For a closed convex set $\mathcal{X}\subset\mathbb{R}^d$, $\Pi_{\mathcal{X}}(z)$ denotes the Euclidean projection of $z$ onto $\mathcal{X}$. Given a filtration ${\mathcal{F}_t}$, a sequence $\{X_t\}$ is $\{\mathcal{F}_t\}$-adapted if $X_t$ is $\mathcal{F}_t$-measurable, and is predictable if $X_t$ is $\mathcal{F}_{t-1}$-measurable. Conditional expectations are written $\mathbb{E}[\cdot \mid \mathcal{F}_{t-1}]$. We write $\inf\varnothing=\infty$. For two nonnegative sequences $\{a_t\}$ and $\{b_t\}$, we write $a_t=O(b_t)$ if there exists a constant $C<\infty$ such that $a_t\le C,b_t$ for all $t$; we write $a_t\lesssim b_t$ synonymously with $a_t=O(b_t)$, and $a_t\asymp b_t$ if $a_t\lesssim b_t$ and $b_t\lesssim a_t$. Unless stated otherwise, implicit constants in $O(\cdot)$, $\lesssim$, and $\asymp$ may depend on fixed problem parameters but are independent of the time index $t$ (and independent of $\varepsilon$ in stopping-time bounds). Throughout, probabilities and expectations are taken with respect to the randomness in the stochastic gradients.

%--------------------------------------------------------------------------------
%--------------------------------------------------------------------------------
%--------------------------------------------------------------------------------

\section{Anytime-Valid Confidence Sequences}
\label{sec:anytime:CS}

In this section we develop anytime-valid confidence sequences for stochastic gradient descent that remain valid under continuous monitoring and arbitrary data-dependent stopping. The constructions are \emph{trajectory-adaptive}, depending on quantities generated along the SGD path rather than on fixed global problem parameters. We treat the convex and nonconvex settings separately. In the convex case, we derive confidence sequences for weighted average suboptimality of projected SGD. In the nonconvex case, we obtain time-uniform certificates for weighted first-order stationarity. Both constructions admit fully observable variants suitable for online use.

%--------------------------------------------------------------------------------

\subsection{Convex case: suboptimality certificates}
\label{subsec:CS:convex}

In this section we derive an adaptive, anytime-valid confidence sequence for the weighted average suboptimality of projected SGD. Our goal is to construct, for each confidence level $\alpha\in(0,1)$, an upper confidence sequence $\{U_t(\alpha)\}$ in the sense of Definition~\ref{def:cs} such that $\{\forall t\ge 1:\ \bar F_t \le U_t(\alpha)\}$ holds with probability $1-\alpha$, where $\bar F_t$ is the weighted suboptimality defined in \eqref{eq:weighted:subopt}. A central feature of our construction is \emph{adaptivity}: the bound adjusts to the realized optimization trajectory through the distances $\|x_t - x^\star\|$ and the squared stochastic gradient norms $\|g_t\|^2$, rather than relying solely on global parameters such as $R_x$ and $G$. A simple corollary will later convert this adaptive bound into a fully observable version involving only $R_x$, the realized quantities $\|g_t\|^2$, and the stepsizes.

To state the main result, we introduce the squared distance process
\[
    Z_t := \|x_t - x^\star\|^2,
\]
and a trajectory-dependent quantity that enters the width of the confidence sequence: for each $t\ge 1$, define
\begin{equation}
\label{eq:sigma-Sigma-def}
    \sigma_t^2 := \sigma^2 \eta_t^2 Z_t,
    \qquad
    \Sigma_t^2 := \sum_{s=1}^t \sigma_s^2,
    \qquad
    \Sigma_{t,\mathrm{eff}}^2
    :=
    \max\!\left\{
        \Sigma_t^2,\;
        2\Big(\log \tfrac{2}{\alpha}+1\Big)
    \right\},
\end{equation}
where $\sigma^2 > 0$ is the sub-Gaussian variance proxy from Assumption~\ref{as:gradients}(ii). The quantity $\Sigma_{t,\mathrm{eff}}^2$ will be referred to as the \emph{effective cumulative variance proxy}. It coincides with the true cumulative variance proxy $\Sigma_t^2$ once the latter exceeds a fixed threshold, and serves to prevent small-variance degeneracies in the time-uniform concentration argument. Larger values of $\Sigma_{t,\mathrm{eff}}^2$ correspond to trajectories with larger stepsizes or larger distances to the optimum, and therefore lead to wider confidence sequences. Although $Z_t$ is not directly observable, it is always bounded by $R_x^2$, and keeping it explicit highlights which components of the trajectory influence the final bound. A fully observable version of the confidence sequence is provided in Corollary~\ref{cor:observable-CS} by replacing $Z_t$ with $R_x^2$.

Our confidence sequence will be stated in terms of the following data-dependent upper boundary:
\begin{equation}
\label{eq:U-def}
    U_t(\alpha)
    :=
    \frac{1}{2S_t}
    \Bigg(
        7\,\sqrt{\Sigma_{t,\mathrm{eff}}^2\Big(\log\tfrac{2}{\alpha} + \log\log(\mathrm e + \Sigma_{t,\mathrm{eff}}^2)\Big)}
        + Z_1
        + \sum_{s=1}^t \eta_s^2 \|g_s\|^2
    \Bigg).
\end{equation}
This boundary combines three sources of variability: the effective cumulative variance proxy $\Sigma_{t,\mathrm{eff}}^2$, the initial distance $Z_1$, and the accumulated squared gradient steps $\sum_{s=1}^t \eta_s^2\|g_s\|^2$. The leading term involving $\Sigma_{t,\mathrm{eff}}^2$ captures the stochastic fluctuations of the process and stems from a time-uniform concentration inequality applied to an adapted sequence satisfying conditional sub-Gaussian tail bounds (see the proof of Theorem~\ref{thm:anytime-suboptimality}). In contrast, the terms $Z_1$ and $\sum_{s=1}^t \eta_s^2\|g_s\|^2$ come directly from the structure of projected SGD, and appear when the distance recursion is telescoped. The next theorem establishes that $U_t(\alpha)$ forms an anytime-valid upper confidence sequence for the weighted suboptimality process $\{\bar F_t\}_{t\ge 1}$.

\begin{theorem}[Adaptive anytime-valid suboptimality bound]
\label{thm:anytime-suboptimality}
Suppose Assumptions~\ref{as:stepsizes} and \ref{as:gradients} hold. Then, for every $\alpha\in(0,\tfrac{2}{\mathrm{e}})$, the process $\{U_t(\alpha)\}_{t\ge 1}$ defined in~\eqref{eq:U-def} satisfies
\[
    \mathbb{P}\!\left(
        \forall t\ge 1:\ \bar F_t \le U_t(\alpha)
    \right)
    \;\ge\; 1-\alpha,
\]
where $\bar F_t$ is the weighted suboptimality in \eqref{eq:weighted:subopt}. In particular, $\{U_t(\alpha)\}_{t\ge 1}$ is an anytime-valid upper confidence sequence for $\{\bar F_t\}_{t\ge 1}$ in the sense of Definition~\ref{def:cs}.
\end{theorem}

Before turning to the proof, we record an immediate consequence of Theorem~\ref{thm:anytime-suboptimality}. The theorem itself produces an adaptive confidence sequence that depends on the (unobservable) distances $\|x_t - x^\star\|$ and on the realized stochastic gradients. In practice, one often wishes to work with fully observable bounds, expressed only in terms of the known constraint radius $R_x$, the realized squared gradient norms $\|g_t\|^2$, and the predictable stepsizes $\{\eta_t\}$. The following corollary provides such a variant.

\begin{corollary}[Observable anytime-valid suboptimality bound]
\label{cor:observable-CS}
Under the conditions of Theorem~\ref{thm:anytime-suboptimality}, suppose in addition that $\mathcal{X}$ satisfies $\|x - x^\star\| \le R_x$ for all $x \in \mathcal{X}$. Define
\[
    \widetilde{\sigma}_t^2 := \sigma^2 \eta_t^2 R_x^2,
    \qquad
    \widetilde{\Sigma}_t^2 := \sum_{s=1}^t \widetilde{\sigma}_s^2
    = \sigma^2 R_x^2 V_t,
    \qquad
    \widetilde{\Sigma}_{t,\mathrm{eff}}^2
    :=
    \max\!\left\{
        \widetilde{\Sigma}_t^2,\;
        2\Big(\log \tfrac{2}{\alpha}+1\Big)
    \right\}.
\]
Then for every $\alpha\in(0,\tfrac{2}{\mathrm{e}})$ the process
\[
    U_t^{\mathrm{obs}}(\alpha)
    :=
    \frac{1}{2 S_t}
    \Bigg(
        7\sqrt{
            \widetilde{\Sigma}_{t,\mathrm{eff}}^2\!
            \left(
                \log\tfrac{2}{\alpha}
                + \log\log(\mathrm e + \widetilde{\Sigma}_{t,\mathrm{eff}}^2)
            \right)}
        + R_x^2
        + \sum_{s=1}^t \eta_s^2 \|g_s\|^2
    \Bigg)
\]
is an anytime-valid upper confidence sequence for $\{\bar F_t\}_{t\ge 1}$.
\end{corollary}

Note that $U_t^{\mathrm{obs}}(\alpha)$ is $\mathcal{F}_t$-measurable and can be computed online at each iteration using only observable information. The following remark shows that the $1/S_t$ normalization appearing in the bound
is unavoidable from a time-uniform perspective, even in simple deterministic
settings. While the numerator in $U_t(\alpha)$ may exhibit different behavior
depending on the stepsize regime and the realized stochastic gradients,
no anytime-valid confidence sequence can shrink uniformly faster than order
$1/S_t$ in general.

\begin{remark}[Unavoidability of the $1/S_t$ normalization]
\label{rem:lower-1d}
The $1/S_t$ \emph{normalization} in Theorem~\ref{thm:anytime-suboptimality} is, in general, unavoidable from a time-uniform (anytime-valid) perspective, even in a very simple one-dimensional deterministic setting. Fix any stepsizes $\{\eta_t\}_{t\ge 1}$ with $\eta_t>0$, $\sum_t \eta_t = \infty$, and $\sum_t \eta_t^2 < \infty$, and consider the quadratic model
\[
    f(x) := \frac{\mu}{2}x^2,\qquad \mu\in(0,1),
\]
with minimizer $x^\star = 0$, $f(x^\star) = 0$, and deterministic oracle $g_t = \nabla f(x_t) = \mu x_t$.  Starting from $x_1 = x_0 \neq 0$, the projected SGD recursion reduces to the deterministic update $x_{t+1} = x_t - \eta_t g_t$, and the weighted cumulative suboptimality
\[
    A_t := \sum_{s=1}^t \eta_s \Big(f(x_s)-f(x^\star)\Big)
\]
is nondecreasing with
\[
    A_t \;\ge\; A_1
    = \eta_1 \frac{\mu}{2}x_1^2 \;>\; 0
    \qquad\text{for all } t\ge 1.
\]
For any anytime-valid upper confidence sequence $\{U_t(\alpha)\}_{t\ge 1}$ over a model class containing this example, validity on this deterministic path forces $A_t \le S_t U_t(\alpha)$ for all $t$, and hence
\[
    S_t U_t(\alpha) \;\ge\; A_1
    \qquad\text{for all } t\ge 1.
\]
Thus, even in this benign one-dimensional model, time-uniform validity forces $U_t(\alpha)$ to scale no better than order $1/S_t$ along this path. This example illustrates that the appearance of the factor $1/S_t$ in time-uniform (anytime-valid) upper bounds is necessary. \hfill $\clubsuit$
\end{remark}

The following remark clarifies that the assumption $\|x - x^\star\| \le R_x$
for all $x \in \mathcal X$ is sufficient but not necessary for constructing an
observable confidence sequence.

\begin{remark}[Compactness of $\mathcal X$]
\label{rmk:compactness}
The compactness of $\mathcal X$ is \emph{not} required for the validity of Theorem~\ref{thm:anytime-suboptimality}, which holds without any boundedness assumption on the iterates. Compactness is used only to obtain the \emph{observable} confidence sequence in Corollary~\ref{cor:observable-CS}, where a uniform bound $R_x$ on the distances $\|x_t - x^\star\|$ along the SGD trajectory is needed. Projection onto a compact set provides a simple and explicit mechanism for ensuring such a bound. However, both compactness and projection can be omitted whenever an \emph{a priori} bound $\|x_t - x^\star\|\le R_x$ is known to hold along the trajectory, for instance under stability conditions that guarantee the iterates remain in a bounded region. \hfill $\clubsuit$
\end{remark}

%--------------------------------------------------------------------------------

\subsection{Nonconvex case: stationarity certificates}
\label{subsec:CS:nonconvex}

In this subsection we extend the confidence-sequence framework to nonconvex SGD, where progress is quantified through first-order stationarity. Throughout this section, we consider \emph{unconstrained} optimization over $\mathcal{X}=\mathbb{R}^d$ with a differentiable, globally smooth, and nonnegative objective, so that the SGD recursion takes the standard form without projection. For a given confidence level $\alpha\in(0,1)$, our goal is to construct a time-uniform upper confidence sequence $\{W_t(\alpha)\}_{t\ge1}$ for the weighted average stationarity measure $\bar G_t$ defined in~\eqref{eq:def-Gbar}. By Definition~\ref{def:cs}, this ensures that the bound $\bar G_t \le W_t(\alpha)$ holds simultaneously for all $t\ge1$ with probability at least $1-\alpha$, even under continuous monitoring and adaptive stopping.

As in the convex case, the construction is \emph{trajectory-adaptive}, but the source of adaptivity differs. In the nonconvex setting, the width of the confidence sequence is driven by the realized gradient norms rather than distances to an optimizer. This reflects the geometry of nonconvex optimization, where stationarity replaces optimality as the relevant notion of progress. Consequently, the certificate tightens automatically when SGD enters low-gradient regions and becomes more conservative along trajectories exhibiting larger stochastic variation or larger gradients.

To state the main result, we introduce a trajectory-dependent variance proxy that governs the width of the confidence sequence. For each $t \ge 1$, define
\begin{equation}
\label{eq:sigma-Sigma-def:nonconvex}
    \nu_t^2 := \sigma^2 \eta_t^2 \|\nabla f(x_t)\|^2,
    \qquad
    \Gamma_t^2 := \sum_{s=1}^t \nu_s^2,
    \qquad
    \Gamma_{t,\mathrm{eff}}^2
    :=
    \max\!\left\{
        \Gamma_t^2,\;
        2\Big(\log \tfrac{2}{\alpha}+1\Big)
    \right\}.
\end{equation}
Here $\sigma^2>0$ is the sub-Gaussian variance proxy from Assumption~\ref{as:gradients}. As in the convex case, we refer to $\Gamma_{t,\mathrm{eff}}^2$ as the \emph{effective cumulative variance proxy}. The effective version coincides with the true cumulative variance $\Gamma_t^2$ once the latter exceeds a fixed threshold, and serves to prevent small-variance degeneracies in the construction of uniform confidence boundaries.

The quantity $\Gamma_t^2$ depends directly on the squared gradient norms encountered along the trajectory and on the stepsizes, and thus captures the intrinsic stochastic scale of the observed SGD path. Larger gradients or larger stepsizes lead to larger values of $\Gamma_t^2$ and consequently to wider confidence sequences, while trajectories that enter low-gradient regimes induce tighter bounds. At this stage, $\Gamma_t^2$ is not assumed to be observable, as it involves the true gradients $\nabla f(x_t)$. We keep this dependence explicit in order to highlight which aspects of the optimization trajectory determine the statistical resolution of the certificate. In many applications, one may have access to an a priori bound on $\|\nabla f(x_t)\|$ along the trajectory; in such cases, $\Gamma_t^2$ can be upper bounded by an observable proxy, yielding a fully implementable confidence sequence. We formalize this observable variant in Corollary~\ref{cor:observable-CS:nonconvex}.

Our anytime-valid certificate will be expressed in terms of the following data-dependent upper boundary:
\begin{equation}
\label{eq:U-def:nonconvex}
    W_t(\alpha)
    :=
    \frac{1}{S_t}
    \Bigg(
        4\,\sqrt{\Gamma_{t,\mathrm{eff}}^2
        \Big(\log\tfrac{2}{\alpha}
        + \log\log(\mathrm e+\Gamma_{t,\mathrm{eff}}^2)\Big)}
        +
        f(x_1)
        + \frac{L}{2}\sum_{s=1}^t \eta_s^2\|g_s\|^2
    \Bigg),
\end{equation}
where $\Gamma_{t,\mathrm{eff}}^2$ is the effective cumulative variance proxy defined in~\eqref{eq:sigma-Sigma-def:nonconvex}. The boundary~\eqref{eq:U-def:nonconvex} combines three distinct contributions. The leading square-root term, involving $\Gamma_{t,\mathrm{eff}}^2$, captures the stochastic fluctuations induced by gradient noise and is obtained via a time-uniform concentration inequality for an adapted process satisfying conditional sub-Gaussian tail bounds. The remaining terms arise from the smoothness structure of the objective: the initial function value $f(x_1)$ and the accumulated squared updates $\sum_{s=1}^t \eta_s^2\|g_s\|^2$ appear when the standard smoothness descent inequality for SGD is telescoped along the trajectory. The next theorem establishes that the boundary $W_t(\alpha)$ defined in~\eqref{eq:U-def:nonconvex} forms an anytime-valid upper confidence sequence for the weighted average stationarity measure $\{\bar G_t\}_{t\ge1}$.

\begin{theorem}[Adaptive anytime-valid stationarity bound]
\label{thm:anytime-stationarity:nonconvex}
Suppose Assumptions~\ref{as:stepsizes}, \ref{as:gradients} and \ref{as:smoothness} hold. Then, for every $\alpha\in(0,\tfrac{2}{\mathrm{e}})$, the process $\{W_t(\alpha)\}_{t\ge1}$ defined in~\eqref{eq:U-def:nonconvex} satisfies
\[
    \mathbb{P}\!\left(
        \forall t\ge 1:\ \bar G_t \le W_t(\alpha)
    \right)
    \;\ge\; 1-\alpha.
\]
In particular, $\{W_t(\alpha)\}_{t\ge1}$ is an anytime-valid upper confidence sequence for $\{\bar G_t\}_{t\ge1}$ in the sense of Definition~\ref{def:cs}.
\end{theorem}

Theorem~\ref{thm:anytime-stationarity:nonconvex} shows that the boundary $W_t(\alpha)$ provides a statistically valid certificate for approximate first-order stationarity along the SGD trajectory. Beyond its validity, the certificate admits a natural interpretation: its evolution reflects the interaction between the stepsize sequence, the magnitude of stochastic gradient noise, and the realized optimization path. Depending on these factors, the bound may decrease, stabilize, or fluctuate, thereby tracking the finest level of stationarity that can be certified from the observed data at each iteration. In this sense, the confidence sequence quantifies not only algorithmic progress, but also the intrinsic statistical resolution imposed by noise and stepsize choices. In the next corollary, we specialize the construction to obtain a fully observable variant under an explicit bound on the gradient norms, yielding an online-implementable confidence sequence.

\begin{corollary}[Observable anytime-valid stationarity bound]
\label{cor:observable-CS:nonconvex}
Under the conditions of Theorem~\ref{thm:anytime-stationarity:nonconvex}, suppose in addition that there exists a constant $G\in(0,\infty)$ such that
\[
    \|\nabla f(x_t)\| \le G
    \qquad \text{almost surely } \forall t\ge 1.
\]
Define the observable variance proxy
\[
    \widetilde{\nu}_t^2 := \sigma^2 \eta_t^2 G^2,
    \qquad
    \widetilde{\Gamma}_t^2 := \sum_{s=1}^t \widetilde{\nu}_s^2
    = \sigma^2 G^2 V_t,
    \qquad
    \widetilde{\Gamma}_{t,\mathrm{eff}}^2
    :=
    \max\!\left\{
        \widetilde{\Gamma}_t^2,\;
        2\Big(\log \tfrac{2}{\alpha}+1\Big)
    \right\}.
\]
Then, for every $\alpha\in(0,\tfrac{2}{\mathrm{e}})$, the process
\[
    W_t^{\mathrm{obs}}(\alpha)
    :=
    \frac{1}{S_t}
    \Bigg(
        4\,\sqrt{
            \widetilde{\Gamma}_{t,\mathrm{eff}}^2
            \Big(
                \log\tfrac{2}{\alpha}
                + \log\log(\mathrm e + \widetilde{\Gamma}_{t,\mathrm{eff}}^2)
            \Big)}
            +
        f(x_1)
        + \frac{L}{2}\sum_{s=1}^t \eta_s^2\|g_s\|^2
    \Bigg)
\]
is an anytime-valid upper confidence sequence for the weighted average stationarity measure $\{\bar G_t\}_{t\ge 1}$.
\end{corollary}

In particular, provided an a priori bound $G$ is available, $W_t^{\mathrm{obs}}(\alpha)$ is $\mathcal{F}_t$-measurable and can be evaluated online at each iteration using only quantities observable along the SGD trajectory.

%--------------------------------------------------------------------------------

\subsection{Rates under common stepsize choices}
\label{subsec:behavior-CS}

We now analyze how the observable anytime-valid confidence sequences evolve under standard stepsize schedules for stochastic gradient descent, focusing on stochastic approximation stepsizes satisfying $\sum_{t=1}^\infty \eta_t=\infty$ and $\sum_{t=1}^\infty \eta_t^2<\infty$ (e.g., of the form $\eta_t = \eta_0 t^{-\gamma}$ with $\gamma\in(1/2,1]$), and on polynomial stepsizes of the form $\eta_t = \eta_0 t^{-1/2}$. For these regimes, we establish explicit decay rates for the confidence sequence, both in expectation and almost surely along the realized SGD trajectory. These results clarify how the choice of stepsize governs the asymptotic magnitude of the certificate, while certified stopping rules and stopping-time guarantees are developed separately in Section~\ref{sec:stopping}.

Throughout this subsection, we fix a confidence level $\alpha\in(0,2/\mathrm e)$ and focus on the dependence of the confidence sequences on the iteration index $t$. The constants implicit in the $O(\cdot)$ notation may depend on $\alpha$, but not on $t$; see Remark~\ref{rem:alpha-dependence} for a discussion of the explicit dependence on $\alpha$.

\begin{assumption}[Stochastic gradients II]
\label{as:gradients:moment}
For all $t\ge 1$, the stochastic gradient $g_t$ satisfies
\[
    \mathbb{E}\big[\|g_t\|^2 \mid \mathcal{F}_{t-1}\big]
    \le M^2
    \qquad \text{almost surely},
\]
for some constant $M\in(0,\infty)$.
\end{assumption}

Assumption~\ref{as:gradients:moment} is standard and is introduced here only to simplify the interpretation of the confidence sequences under different stepsize regimes; it is not required for the anytime-valid guarantees proved earlier. It is also consistent with Assumption~\ref{as:gradients}(ii): if $\xi_t = g_t-\nabla f(x_t)$ is conditionally sub-Gaussian with variance proxy $\sigma^2$, then a standard consequence of \eqref{eq:cond-subg-noise} is that
\[
    \mathbb{E}\!\left[\|g_t\|^2 \mid \mathcal{F}_{t-1}\right]
    = \|\nabla f(x_t)\|^2
    + \mathbb{E}\!\left[\|\xi_t\|^2 \mid \mathcal{F}_{t-1}\right]
    \le \|\nabla f(x_t)\|^2 + d\sigma^2
    \qquad \text{almost surely}.
\]
Thus, whenever the gradient norms are uniformly bounded along the trajectory, Assumption~\ref{as:gradients:moment} holds automatically. In particular, in the nonconvex case, where such boundedness is already imposed to obtain a fully observable confidence sequence, this assumption is without loss of generality.

Before turning to specific stepsize regimes, we note that the \emph{observable} confidence sequences from Corollaries~\ref{cor:observable-CS} and \ref{cor:observable-CS:nonconvex} have the same structural form and the same dependence on the stepsizes. In both the convex and nonconvex settings, the bound at time $t$ can be written as $1/S_t$ times a sum of three terms: a time-uniform term governed by the intrinsic variance scale $V_t=\sum_{s=1}^t \eta_s^2$, a cumulative term of the form $\sum_{s=1}^t \eta_s^2\|g_s\|^2$ arising directly from the SGD recursion, and an initial-condition term depending on $R_x^2$ or $f(x_1)$. At the level of the observable corollaries, the variance proxies entering the time-uniform terms are identical, since both rely on deterministic bounds on $\|x_t-x^\star\|^2$ in the convex case and on $\|\nabla f(x_t)\|^2$ in the nonconvex case. Consequently, the dependence of the confidence sequences on the stepsize schedule is the same in both settings. For this reason, and to avoid repetition, we state the results below only for the confidence sequence $U_t^{\mathrm{obs}}(\alpha)$; the corresponding statements for the nonconvex certificate $W_t^{\mathrm{obs}}(\alpha)$ follow verbatim, with identical stepsize dependence.

\bigskip
\noindent\textbf{Stochastic approximation stepsizes.}
The cleanest characterization of the behavior of the confidence sequence arises under classical \emph{stochastic approximation} stepsizes, in which $\eta_t>0$ and
\begin{equation}
\label{eq:stepsizes:SA}
    \sum_{t=1}^\infty \eta_t=\infty,
    \qquad
    \sum_{t=1}^\infty \eta_t^2<\infty.
\end{equation}
This is the canonical regime underlying the asymptotic theory of stochastic approximation and almost sure convergence results for SGD in convex optimization. In our context, these conditions imply that the normalization $S_t$ diverges, while the intrinsic variance scale in the confidence sequence remains controlled. The following proposition formalizes this intuition by showing that, under stochastic approximation stepsizes, the confidence sequence decays at order $1/S_t$, both almost surely and in expectation.

\begin{proposition}[Rate-controlled decay under stochastic approximation stepsizes]
\label{prop:Ut-SA}
Under the conditions of Corollary~\ref{cor:observable-CS}, suppose in addition that Assumption~\ref{as:gradients:moment} holds and that the stepsizes satisfy \eqref{eq:stepsizes:SA}. Fix any $\alpha\in(0,\tfrac{2}{\mathrm e})$. Then:
\begin{enumerate}
    \item[\textup{(i)}] The observable confidence sequence satisfies
    \[
        \mathbb{E}\!\left[U_t^{\mathrm{obs}}(\alpha)\right] = O\left(\frac{1}{S_t}\right).
    \]
    \item[\textup{(ii)}] Moreover,
    \[
        U_t^{\mathrm{obs}}(\alpha)=O\left(\frac{1}{S_t}\right)
        \quad \text{almost surely},
    \]
    and in particular $U_t^{\mathrm{obs}}(\alpha)\to 0$ almost surely as $t\to\infty$.
\end{enumerate}
\end{proposition}

An explicit constant in the bound $\mathbb{E}[U_t^{\mathrm{obs}}(\alpha)] = O(1/S_t)$ can be extracted from the proof. Taken together, Proposition~\ref{prop:Ut-SA} and Remark~\ref{rem:lower-1d} show that under stochastic approximation stepsizes, anytime-valid certification comes at essentially no cost to the rate: the certificate decays as $1/S_t$, and this time-uniform scaling is unimprovable in general.

\begin{remark}[Why no $\log\log \mathrm t$ term?]
A useful consequence of the square-summable stepsize condition is that the iterated-logarithm correction in $U_t^{\mathrm{obs}}(\alpha)$ is expressed in terms of the intrinsic variance scale $\Sigma_{t,\mathrm{eff}}^2$ rather than the time index $t$. In classical boundary-crossing problems, the predictable variance typically grows linearly with time, leading to $\log\log t$ corrections \cite{robbins1970boundary,howard2021time}. Here, however, $\sum_{t=1}^\infty \eta_t^2<\infty$ forces $\Sigma_{t,\mathrm{eff}}^2$ to remain uniformly bounded, so the variance-based iterated logarithm is itself bounded. This places stochastic approximation in a regime where variance-adaptive confidence sequences can be substantially tighter than bounds written purely as a function of time; see, for example, \cite{de2009self} and the discussion in \cite[Section~3]{howard2021time}. \hfill $\clubsuit$
\end{remark}

\bigskip
\noindent\textbf{Polynomial stepsizes $\eta_t=\eta_0\, t^{-1/2}$.}
Beyond stochastic approximation stepsizes, an additional regime of particular interest consists of polynomially decaying stepsizes of the form $\eta_t=\eta_0\, t^{-1/2}$. This schedule can be viewed as the horizon-free analogue of the classical finite-horizon tuning $\eta_t=\eta_0\,T^{-1/2}$, which is held constant over $t=1,\ldots,T$ when a terminal time $T$ is specified in advance. We therefore use $\eta_t=\eta_0\,t^{-1/2}$ as a natural benchmark for comparing our anytime-valid guarantees against the familiar fixed-horizon rates in stochastic optimization.

In contrast to stochastic approximation stepsizes, for $\eta_t=\eta_0\, t^{-1/2}$ the variance scale $V_t=\sum_{s=1}^t \eta_s^2$ grows logarithmically. As a consequence, although the observable confidence sequence still decays along the realized trajectory, its pathwise behavior incurs an additional logarithmic factor through the variance scale $V_t$. We now formalize this behavior by characterizing the decay of the confidence sequence both in expectation and almost surely.

\begin{proposition}[Rate-controlled decay under polynomial stepsizes]
\label{prop:Ut-poly}
Under the conditions of Corollary~\ref{cor:observable-CS}, suppose in addition that Assumption~\ref{as:gradients:moment} holds and that the stepsizes satisfy $\eta_t=\eta_0\,t^{-1/2}$. Fix any $\alpha\in(0,\tfrac{2}{\mathrm e})$. Then:
\begin{enumerate}
    \item[\textup{(i)}] The observable confidence sequence satisfies
    \[
        \mathbb{E}\!\left[U_t^{\mathrm{obs}}(\alpha)\right]
        = O\left(\frac{\log t}{\sqrt{t}}\right).
    \]
    \item[\textup{(ii)}] If, in addition, the gradient norms are uniformly bounded,
    $\|\nabla f(x_t)\|\le G$ almost surely for all $t\ge1$, then
    \[
        U_t^{\mathrm{obs}}(\alpha)
        = O\left(\frac{\log t}{\sqrt{t}}\right)
        \quad \text{almost surely},
    \]
    and in particular $U_t^{\mathrm{obs}}(\alpha)\to 0$ almost surely as $t\to\infty$.
\end{enumerate}
\end{proposition}

An explicit constant in the bound $\mathbb{E}[U_t^{\mathrm{obs}}(\alpha)] = O(\log t / \sqrt{t})$ can be extracted from the proof. The uniform boundedness assumption on the gradients in part~\textup{(ii)} is mild in the present setting. In the nonconvex case, such a bound is already imposed in order to obtain a fully observable confidence sequence. In the convex case, when the iterates are projected onto a compact set $\mathcal X$, this condition is automatically satisfied whenever $f$ is Lipschitz on $\mathcal X$.

\begin{remark}[On the $\log t$ term]
\label{rem:logt}
The logarithmic factor in Proposition~\ref{prop:Ut-poly} comes from using a single, horizon-independent stepsize schedule. In classical \emph{finite-horizon} analyses, one fixes a terminal time $T$ and tunes the stepsize as $\eta_t=\eta_0\,T^{-1/2}$ (constant in $t$), which gives $\sum_{s=1}^T \eta_s^2=\eta_0^2$ and thus avoids logarithmic growth. In contrast, an anytime-valid guarantee must rely on a single stepsize sequence that is defined and usable for all times, such as $\eta_t=\eta_0\,t^{-1/2}$, for which $\sum_{s=1}^t \eta_s^2 \asymp \log t$. Hence the $\log t$ term should be viewed as the cost of horizon-free operation under the square-root schedule, rather than as a cost intrinsic to the confidence-sequence methodology. \hfill $\clubsuit$
\end{remark}

\begin{remark}[Dependence on the confidence level $\alpha$]
\label{rem:alpha-dependence}
We briefly comment on the dependence on the confidence level $\alpha$, which is suppressed in the rate statements above. For simplicity, suppose that for the timescales of interest the effective variance proxy satisfies $\widetilde{\Sigma}_{t,\mathrm{eff}}^2=\sigma^2R_x^2V_t$; that is, the maximum in the definition of $\widetilde{\Sigma}_{t,\mathrm{eff}}^2$ is attained by the intrinsic variance scale rather than by the lower truncation term $2(\log ({2}/{\alpha})+1)$. In this typical regime, the dependence on $\alpha$ enters only through the boundary term $\log(2/\alpha)$. Under stochastic approximation stepsizes, $V_t$ is bounded, and the proof yields
\[
    \mathbb{E}\!\left[U_t^{\mathrm{obs}}(\alpha)\right]
    =
    O\left(\frac{\sqrt{\log(2/\alpha)}}{S_t}\right),
\]
up to constants independent of $t$. For polynomial stepsizes $\eta_t=\eta_0\, t^{-1/2}$, we have $V_t\asymp \log t$, leading to
\[
    \mathbb{E}\!\left[U_t^{\mathrm{obs}}(\alpha)\right]
    =
    O\left(
    \frac{\sqrt{\log t\,\big(\log(2/\alpha)+\log\log t\big)}}{\sqrt t}
    +
    \frac{\log t}{\sqrt t}
    \right),
\]
again up to constants independent of $t$. For any fixed $\alpha\in(0,2/\mathrm e)$, these expressions recover the $t$-rates stated in Propositions~\ref{prop:Ut-SA} and~\ref{prop:Ut-poly}.  \hfill $\clubsuit$
\end{remark}

%--------------------------------------------------------------------------------
%--------------------------------------------------------------------------------
%--------------------------------------------------------------------------------

\section{Anytime-Valid Stopping Rules}
\label{sec:stopping}

A central motivation for our confidence-sequence approach is to turn SGD into an \emph{anytime-valid} optimization procedure: instead of committing to a fixed horizon in advance, the user would like to stop the algorithm based on the information actually observed along the trajectory. In practice, most stopping rules for SGD are heuristic (e.g., they monitor validation loss, gradient norms, or empirical convergence patterns) and existing guarantees typically hold only at a pre-specified time or for a fixed, non-adaptive stopping rule. In contrast, a confidence sequence is designed precisely to support \emph{optional stopping}: it provides a high-probability guarantee that remains valid under arbitrary data-dependent monitoring and stopping strategies.

As in Section~\ref{subsec:behavior-CS}, we present the stopping rules and their analysis only for the convex confidence sequence from Corollary~\ref{cor:observable-CS}, to avoid repetition. At the level of the observable confidence sequences, the convex and nonconvex bounds have identical structure and identical dependence on the stepsize schedule, and differ only in the interpretation of the certified performance measure. In particular, all of the stopping rules, finiteness arguments, and stopping-time complexity bounds stated below carry over verbatim to the nonconvex confidence sequence from Corollary~\ref{cor:observable-CS:nonconvex}, with $\varepsilon$-optimality replaced by $\varepsilon$-stationarity; i.e., with $f(\bar x_t)-f(x^\star)$ replaced by $\min_{1 \le s \le t} \|\nabla f(x_s)\|^2$.

Corollary~\ref{cor:observable-CS} provides an observable anytime-valid upper confidence sequence $\{U_t^{\mathrm{obs}}(\alpha)\}_{t\ge 1}$ for the weighted average suboptimality $\bar F_t$ in \eqref{eq:weighted:subopt}. Combining this with the convexity bound $f(\bar x_t)-f(x^\star) \le \bar F_t$, we obtain, for any fixed $\alpha\in(0,\tfrac{2}{\mathrm e})$,
\begin{equation}
\label{eq:cs-stopping-recall}
    \mathbb{P}\!\left(
        \forall t\ge 1:\ 
        f(\bar x_t)-f(x^\star) \le U_t^{\mathrm{obs}}(\alpha)
    \right)
    \;\ge\; 1-\alpha,
\end{equation}
where $\bar x_t = S_t^{-1}\sum_{s=1}^t \eta_s x_s$ is the weighted average iterate, with $S_t=\sum_{s=1}^t \eta_s$.

Given a target accuracy $\varepsilon>0$ and confidence level $\alpha\in(0,\tfrac{2}{\mathrm e})$, the time-uniform guarantee \eqref{eq:cs-stopping-recall} suggests a simple data-dependent stopping rule: monitor the observable certificate $U_t^{\mathrm{obs}}(\alpha)$ along the run of SGD and stop as soon as it falls below $\varepsilon$.  We formalize this through the stopping time
\begin{equation}
\label{eq:tau-def-stopping}
    \tau_\varepsilon
    :=
    \inf\bigl\{ t\ge 1 : U_t^{\mathrm{obs}}(\alpha) \le \varepsilon \bigr\},
\end{equation}
with the convention $\inf\varnothing=\infty$.  By construction, $\tau_\varepsilon$ depends only on quantities observed along the trajectory and is therefore a stopping time with respect to the natural filtration $\{\mathcal{F}_t\}_{t\ge 1}$.

The following theorem shows that this intuitive rule is in fact statistically sound: whenever the algorithm stops, the returned weighted average iterate satisfies the desired accuracy guarantee with confidence $1-\alpha$, without requiring the stopping time to be fixed in advance.

\begin{theorem}[Certified anytime-valid stopping rule]
\label{thm:stopping-certified-main}
Suppose Assumptions~\ref{as:stepsizes} and \ref{as:gradients} hold, and let $\{U_t^{\mathrm{obs}}(\alpha)\}_{t\ge 1}$ be the observable anytime-valid upper confidence sequence from Corollary~\ref{cor:observable-CS}. Fix $\alpha\in(0,\tfrac{2}{\mathrm e})$, and define $\tau_\varepsilon$ as in
\eqref{eq:tau-def-stopping}. Then
\begin{equation}
\label{eq:stopping-certified-direct}
\mathbb{P}\!\left(
    \{\tau_\varepsilon<\infty\}\cap\{f(\bar x_{\tau_\varepsilon})-f(x^\star)>\varepsilon\}
\right)
\le \alpha .
\end{equation}
Consequently, with probability at least $1-\alpha$, whenever $\tau_\varepsilon<\infty$ we have $f(\bar x_{\tau_\varepsilon})-f(x^\star) \le \varepsilon$.
\end{theorem}

Theorem~\ref{thm:stopping-certified-main} establishes the essential statistical point: the confidence sequence $U_t^{\mathrm{obs}}(\alpha)$ may be monitored continuously and stopped on without invalidating the confidence guarantee. Thus, whenever the stopping rule triggers, the returned iterate is $\varepsilon$-optimal with confidence $1-\alpha$. We now turn from validity to termination, and ask under what conditions the certificate crosses the threshold $\varepsilon$ in finite time. We address this question for the two canonical stepsize schedules introduced in Section~\ref{subsec:behavior-CS}. Under stochastic approximation stepsizes, the confidence sequence decays at rate $1/S_t$ almost surely, which implies that the stopping time $\tau_\varepsilon$ is almost surely finite under the same conditional second moment assumptions used to control its trajectory. For polynomial stepsizes of the form $\eta_t=\eta_0\,t^{-1/2}$, the intrinsic variance scale grows logarithmically and the cumulative term $\sum_{s=1}^t \eta_s^2\|g_s\|^2$ must be controlled; in this case, almost sure termination is recovered under a mild additional uniform boundedness assumption along the trajectory, as in Proposition~\ref{prop:Ut-poly}.

\begin{proposition}[Certified stopping under common stepsize regimes]
\label{prop:stopping-certified}
Suppose Assumptions~\ref{as:stepsizes}, \ref{as:gradients}, and \ref{as:gradients:moment} hold, and let $\{U_t^{\mathrm{obs}}(\alpha)\}_{t\ge 1}$ be the observable anytime-valid confidence sequence from Corollary~\ref{cor:observable-CS}. Fix $\varepsilon>0$ and $\alpha\in(0,\tfrac{2}{\mathrm e})$, and define $\tau_\varepsilon$ as in~\eqref{eq:tau-def-stopping}. Then:
\begin{enumerate}
    \item[\textup{(i)}] Under the stochastic approximation stepsize conditions \eqref{eq:stepsizes:SA}, the stopping time $\tau_\varepsilon$ is almost surely finite.

    \item[\textup{(ii)}] If $\eta_t=\eta_0\,t^{-1/2}$ and, in addition, the gradient norms are uniformly bounded, $\|\nabla f(x_t)\|\le G$ almost surely for all $t\ge1$, then the stopping time $\tau_\varepsilon$ is again almost surely finite.
\end{enumerate}
Therefore, in both cases, the stopped iterate satisfies the certification guarantee
\begin{equation}\label{eq:stopped-eps-opt}
    \mathbb{P}\Big(
        f(\bar x_{\tau_\varepsilon}) - f(x^\star) \le \varepsilon
    \Big)
    \;\ge\; 1-\alpha.
\end{equation}
\end{proposition}

Proposition~\ref{prop:stopping-certified} establishes that the stopping rule $\tau_\varepsilon$ is both statistically valid and almost surely finite under the standard stepsize regimes considered in this paper. Equipped with this rule, projected SGD becomes a fully implementable, data-dependent $\varepsilon$-optimization procedure: one may monitor $U_t^{\mathrm{obs}}(\alpha)$ online and stop as soon as it falls below $\varepsilon$, without specifying a horizon in advance or invalidating the $1-\alpha$ guarantee.

%--------------------------------------------------------------------------------

\subsection{Stopping-time complexity}
\label{subsec:stopping-time-complexity}

Having established validity and almost sure termination, we now quantify the induced \emph{stopping-time complexity} via the expected stopping time $\mathbb E[\tau_\varepsilon]$, which captures the typical runtime of the certified, continuously monitored stopping rule and enables direct comparison with classical expected-iteration guarantees for stochastic gradient methods.

As above, we state results only for the convex certificate $U_t^{\mathrm{obs}}(\alpha)$ and the suboptimality measure $f(\bar x_t)-f(x^\star)$; the nonconvex analog follows verbatim upon replacing $U_t^{\mathrm{obs}}(\alpha)$ and $f(\bar x_t)-f(x^\star)$ by $W_t^{\mathrm{obs}}(\alpha)$ and $\min_{1\le s\le t}\|\nabla f(x_s)\|^2$, respectively. By Theorem~\ref{thm:stopping-certified-main}, on the event $\{\tau_\varepsilon<\infty\}$ the stopping rule returns the weighted-average iterate $\bar x_{\tau_\varepsilon}$ together with an $\varepsilon$-optimality guarantee at confidence level at least $1-\alpha$, without requiring the stopping time to be fixed in advance. We therefore bound $\mathbb{E}[\tau_\varepsilon]$ under the two canonical stepsize regimes studied earlier: stochastic approximation (SA) stepsizes and the square-root schedule $\eta_t=\eta_0\, t^{-1/2}$. Some results below additionally require bounded gradients along the realized trajectory, which already appeared in the nonconvex setting to obtain an observable certificate and for the certified stopping-time analysis under the schedule $\eta_t=\eta_0\,t^{-1/2}$.

\begin{assumption}[Bounded gradients along the trajectory]
\label{as:bounded-grad}
There exists a constant $G\in(0,\infty)$ such that
\[
    \|\nabla f(x_t)\|\le G
    \qquad\text{almost surely for all } t\ge1.
\]
\end{assumption}

We now record a simple reduction that turns an almost sure envelope for the certificate into a bound on the expected certified stopping time. The argument is entirely pathwise: once an almost sure envelope holds along the trajectory, the stopping time is deterministically bounded by the inverse envelope.

\begin{lemma}[From envelope to expected stopping-time bound]
\label{lem:Etau-from-envelope}
Fix $\alpha\in(0,\tfrac{2}{\mathrm e})$ and $\varepsilon>0$. Suppose there exist a random variable $K\ge0$ and a deterministic function $b:\mathbb{N}\to(0,\infty)$ such that:
\begin{enumerate}
\item[\textup{(i)}] $b(t)$ is nonincreasing in $t$, and $\lim_{t\to\infty}b(t)=0$;
\item[\textup{(ii)}] for all $t\ge1$, $U_t^{\mathrm{obs}}(\alpha)\le K\,b(t)$ almost surely.
\end{enumerate}
Define the deterministic inverse profile $T(u):=\inf\{t\ge1:\, b(t)\le u\}$, for $u\ge 0$, with the convention $\inf\emptyset=\infty$. Then $\tau_\varepsilon \le T(\varepsilon/K)$ almost surely, and hence
\[
    \mathbb{E}[\tau_\varepsilon]\le \mathbb{E}\!\left[T\!\left(\frac{\varepsilon}{K}\right)\right].
\]
\end{lemma}

Lemma~\ref{lem:Etau-from-envelope} reduces expected certified runtime to two ingredients: a \emph{deterministic} decay profile $b(t)$ and a \emph{random} prefactor $K$ controlling the pathwise fluctuations of the observable certificate. The inverse map $T(\cdot)$ is typically nonlinear, so converting the almost sure domination $\tau_\varepsilon \le T(\varepsilon/K)$ into an explicit bound on $\mathbb{E}[\tau_\varepsilon]$ requires integrability of the random variable $T(\varepsilon/K)$, and hence moment (or tail) control on $K$. In particular, when $b(t)$ decays polynomially, $T(\varepsilon/K)$ grows polynomially in $K$, and the required moment order is dictated by the stepsize regime.

\bigskip
\noindent\textbf{Polynomial stepsizes $\eta_t=\eta_0\, t^{-1/2}$.}
We first specialize to the square-root schedule $\eta_t=\eta_0\,t^{-1/2}$, which admits a fully explicit expected stopping-time bound under the baseline assumptions. In this regime the deterministic components of $U_t^{\mathrm{obs}}(\alpha)$ scale as $(1+\log t)/\sqrt t$, and the remaining randomness can be captured by a single weighted noise-square martingale. Lemma~\ref{lem:Mt-maximal} provides the moment control needed to combine Lemma~\ref{lem:Etau-from-envelope} with an explicit inversion of the profile $b(t)\asymp (1+\log t)/\sqrt t$, yielding an explicit stopping-time complexity rate for $\mathbb E[\tau_\varepsilon]$.

\begin{theorem}[Stopping-time complexity for $\eta_t=\eta_0\,t^{-1/2}$]
\label{thm:Etau-sqrt}
Suppose Assumptions~\ref{as:stepsizes}, \ref{as:gradients}, and \ref{as:bounded-grad} hold, and let $\eta_t=\eta_0\,t^{-1/2}$. Fix any $\alpha\in(0,\tfrac{2}{\mathrm e})$. Then, for all $\varepsilon\in(0,\tfrac12)$,
\[
    \mathbb{E}[\tau_\varepsilon]
    =
    O\left(\varepsilon^{-2}\log^2\!\Big(\frac{1}{\varepsilon}\Big)\right).
\]
\end{theorem}

In Theorem~\ref{thm:Etau-sqrt}, the $O(\cdot)$ notation hides constants that depend on the confidence level $\alpha$ and the quantities $(\eta_0,R_x,G,\sigma,d)$, but are independent of the tolerance $\varepsilon$. Moreover, the dependence on $\alpha$ enters only through the time-uniform confidence boundary (via $\sqrt{\log(2/\alpha)}$ and the lower truncation in the effective variance proxy), and can be made explicit by tracing these factors through the proof; for any fixed $\alpha\in(0,2/\mathrm e)$ the stated $\varepsilon$-rate is unchanged.

\begin{remark}[On the $\log^2({1}/{\varepsilon})$ term]
\label{rem:log2eps}
The rate in Theorem~\ref{thm:Etau-sqrt} matches the natural inversion of the certificate scale under the square-root schedule $\eta_t=\eta_0\,t^{-1/2}$. In this regime, $V_t=\sum_{s=1}^t \eta_s^2 \asymp \log t$, and the observable certificate admits an almost sure envelope of order $(\log t)/\sqrt t$. Thus, the stopping condition $U_t^{\mathrm{obs}}(\alpha)\le \varepsilon$ heuristically suggests $\tau_\varepsilon$ of order $\varepsilon^{-2}\log^2(1/\varepsilon)$. Theorem~\ref{thm:Etau-sqrt} makes this intuition precise by establishing the corresponding bound on the expected stopping time. Importantly, the logarithmic factors come from the horizon-free square-root stepsize schedule (through $V_t\asymp \log t$ and the inversion of $(\log t)/\sqrt t$), not from the anytime-valid nature of the certificate, whose role is precisely to keep the guarantee robust under continuous monitoring and data-dependent stopping.
\hfill $\clubsuit$
\end{remark}

Before proceeding, we note that under the stronger assumption of uniformly bounded stochastic gradients, $\|g_t\|\le G$ almost surely, the square-root regime admits an envelope with a deterministic prefactor. In that case, one also obtains a pathwise certified stopping-time bound, yielding $\tau_\varepsilon = O\big(\varepsilon^{-2}\log^2(1/\varepsilon)\big)$ almost surely, with explicit constant depending on $(\eta_0,R_x,G,\sigma,d,\alpha)$ but not on $\varepsilon$.

\bigskip
\noindent\textbf{Stochastic approximation stepsizes.}
We now turn to stochastic approximation (SA) stepsizes. Proposition~\ref{prop:Ut-SA} yields a pathwise envelope of the same general form, namely $U_t^{\mathrm{obs}}(\alpha)\le K/S_t$ almost surely for a finite nonnegative random variable $K$. For polynomial SA stepsizes $\eta_t=\eta_0\,t^{-\gamma}$ with $\gamma\in(\tfrac12,1)$, the cumulative weight satisfies $S_t\asymp t^{1-\gamma}$, and hence Lemma~\ref{lem:Etau-from-envelope} implies $\tau_\varepsilon \lesssim (K/\varepsilon)^{1/(1-\gamma)}$ almost surely. Consequently, obtaining an explicit rate for $\mathbb{E}[\tau_\varepsilon]$ reduces to integrability of $K^{p_\gamma}$ with exponent $p_\gamma:=1/(1-\gamma)$ (e.g., $\gamma=3/4$ requires a fourth moment, while $\gamma=0.9$ requires a tenth moment). This dependence is an intrinsic consequence of inverting the polynomial profile $b(t)=1/S_t$, rather than of optional stopping.

Under the conditional sub-Gaussian noise assumption, one may control such moments via the same martingale decomposition strategy that underlies the proof of Theorem~\ref{thm:Etau-sqrt}: the cumulative noise-square term is decomposed into a predictable compensator plus a martingale $M_t$, and martingale maximal/moment inequalities are then used to bound $\sup_{t\ge1}|M_t|$ in $L^{p_\gamma}$. In particular, for each fixed $p\ge2$ one can obtain $\mathbb{E}[K^{p}]<\infty$ by extending the argument of Lemma~\ref{lem:Mt-maximal} to $L^p$. The key point, however, is that these bounds are highly non-uniform in $p$. First, high-order control requires correspondingly high moments of $\|\xi_t\|$; under conditional sub-Gaussianity these moments are finite for every fixed order, but their deterministic upper bounds grow with the order. Second, the constants in martingale maximal and moment inequalities (e.g., Burkholder's square-function inequality \cite[Theorem~2.10]{hall2014martingale} combined with Doob-type maximal inequalities) also increase rapidly with the exponent; in particular, tracking explicit constants typically yields growth on the order of $p^{\,p}$. As $\gamma\uparrow1$, the required moment order $p_\gamma=1/(1-\gamma)$ diverges, and the implicit constants in the resulting bound for $\mathbb{E}[\tau_\varepsilon]$ therefore deteriorate extremely fast. Thus, while for each fixed $\gamma\in(\tfrac12,1)$ one may in principle recover the natural scaling $\mathbb{E}[\tau_\varepsilon]=O\big(\varepsilon^{-1/(1-\gamma)}\big)$, the associated constants become quantitatively fragile.

This behavior contrasts with the square-root schedule analyzed in Theorem~\ref{thm:Etau-sqrt}. There, the inversion of the profile $b(t)\asymp (1+\log t)/\sqrt t$ yields an upper bound of the schematic form $\mathbb{E}[\tau_\varepsilon]\lesssim \varepsilon^{-2} \,\mathbb{E}\big[K^2\log^2(K/\varepsilon)\big]$ (see the proof for the precise statement), so it suffices to control only a moment slightly above~$2$. Indeed, for any $\delta>0$, the condition $\mathbb{E}[K^{2+\delta}]<\infty$ implies $\mathbb{E}\big[K^2\log^2(K)\big]<\infty$, and thus yields a finite bound on $\mathbb{E}[\tau_\varepsilon]$ with constants that remain stable. 

Therefore, the baseline conditional sub-Gaussian model is sufficient to ensure almost sure termination of the certified stopping rule, but in the SA regime it can lead to expected stopping-time bounds whose constants drastically deteriorate as the stepsizes move from $\eta_t=\eta_0\,t^{-1/2}$ toward $\eta_t=\eta_0\,t^{-1}$. To eliminate this issue, one can impose stronger assumptions on the stochastic gradients: if they are uniformly bounded almost surely, then the envelope in Lemma~\ref{lem:Etau-from-envelope} is deterministic. As a result, one obtains \emph{pathwise} certified stopping-time bounds under standard stochastic approximation stepsizes.

\begin{remark}[SA stepsizes and bounded stochastic gradients]
\label{rem:SA-bounded}
Assume $\|g_t\|\le G$ almost surely for all $t\ge1$ and $\sum_{t=1}^\infty \eta_t^2<\infty$. Then there exists a deterministic constant $K_{\mathrm{det}}<\infty$ (depending only on $(G,R_x,\alpha)$ and the stepsizes through $V_\infty:=\sum_{t=1}^\infty \eta_t^2$) such that, almost surely,
\[
    U_t^{\mathrm{obs}}(\alpha)\le \frac{K_{\mathrm{det}}}{S_t}
    \qquad\text{for all } t\ge1.
\]
Consequently, almost surely,
\[
    \tau_\varepsilon
    \le
    \inf\Bigl\{t\ge1:\ S_t \ge \frac{K_{\mathrm{det}}}{\varepsilon}\Bigr\}.
\]
If $\eta_t=\eta_0\,t^{-\gamma}$ with $\gamma\in(1/2,1)$, then Lemma~\ref{lem:poly-sum} yields $S_t \ge \eta_0 (t^{1-\gamma}-1)/(1-\gamma)$ for all $t\ge1$, and hence
\[
    \tau_\varepsilon
    \le
    \left\lceil
    \left(1+\frac{(1-\gamma)K_{\mathrm{det}}}{\eta_0\,\varepsilon}\right)^{\!\frac{1}{1-\gamma}}
    \right\rceil .
\]
In the boundary case $\eta_t=\eta_0\,t^{-1}$, Lemma~\ref{lem:harmonic-sum} gives $S_t \ge \eta_0 \log t$, and therefore
\[
    \tau_\varepsilon
    \le
    \left\lceil \exp\!\left(\frac{K_{\mathrm{det}}}{\eta_0\,\varepsilon}\right)\right\rceil .
\]
\hfill $\clubsuit$
\end{remark}

%--------------------------------------------------------------------------------
%--------------------------------------------------------------------------------
%--------------------------------------------------------------------------------

\section{Proofs of Theorems~\ref{thm:anytime-suboptimality} and~\ref{thm:anytime-stationarity:nonconvex}}
\label{sec:proofs}

In this section we present the main proofs underlying the anytime-valid confidence sequences for stochastic gradient descent established in Theorems~\ref{thm:anytime-suboptimality} and~\ref{thm:anytime-stationarity:nonconvex}. The arguments exploit the structure of the SGD recursion to develop a time-uniform,
anytime-valid statistical view of optimization progress, grounded in sequential inference and nonnegative supermartingale techniques. The convex case is developed in a detailed, step-by-step manner, with intermediate lemmas and explanatory remarks highlighting the main ideas. The nonconvex proof builds on the same probabilistic framework and is therefore presented more directly, focusing on the smoothness-based decomposition and martingale structure specific to stationarity.

\subsection{Proof of Theorem~\ref{thm:anytime-suboptimality}}
\label{subsec:proof:thm:anytime-suboptimality}

We prove Theorem~\ref{thm:anytime-suboptimality} by decomposing the projected SGD recursion into a sequence of conceptual steps. We first construct a one-step process that compares the instantaneous suboptimality to the geometry of projected SGD and the stochastic gradient noise, and show that this process satisfies a trajectory-dependent sub-Gaussian bound. We then derive a time-uniform concentration bound for the resulting martingale sum, after which a simple telescoping argument and normalization by $S_t$ yield an anytime-valid upper confidence sequence for the weighted suboptimality $\bar F_t$.

\bigskip
\noindent
\textbf{Step 1: Constructing a sub-Gaussian process}

We begin by identifying a natural potential function that tracks the progress of projected SGD toward the minimizer. A standard choice in convex optimization is the squared distance to the optimum, $Z_t = \|x_t - x^\star\|^2$. Because the update $x_{t+1} = \Pi_{\mathcal{X}}(x_t - \eta_t g_t)$ is a Euclidean projection onto a closed convex set, the evolution of $Z_t$ can be controlled using the non-expansiveness of projection:
\begin{equation}
\label{eq:proj-nonexpansion}
    Z_{t+1}
    \;=\;
    \|x_{t+1} - x^\star\|^2
    \;\le\;
    \|x_t - \eta_t g_t - x^\star\|^2
    \;=\;
    Z_t - 2\eta_t\inner{x_t - x^\star}{g_t} + \eta_t^2\|g_t\|^2.
\end{equation}
Rearranging this gives the pathwise lower bound $Z_t - Z_{t+1}\ge 2\eta_t\inner{x_t - x^\star}{g_t} - \eta_t^2\|g_t\|^2$. To connect this distance recursion to the suboptimality we care about, we package the suboptimality, the distance decrement, and the squared gradient step into a single quantity
\begin{equation}
\label{eq:Xt-def-proj}
    X_t
    \;:=\;
    2\eta_t\big(f(x_t)-f(x^\star)\big)
    - (Z_t - Z_{t+1})
    - \eta_t^2\|g_t\|^2,
\end{equation}
with $X_0 \coloneqq 0$. The next lemma shows that $X_t$ has two crucial properties: it has nonpositive conditional drift, and its conditional moment generating function is dominated by that of a sub-Gaussian variable with variance proxy on the order of $\sigma_t^2 = \sigma^2 \eta_t^2 Z_t$, as defined in~\eqref{eq:sigma-Sigma-def}. This is where convexity of $f$, the unbiasedness assumption, and the conditional sub-Gaussian noise assumption enter.

\begin{lemma}[Conditional sub-Gaussian control]
\label{lem:mgf-proj-subg}
Let $X_t$ be defined as in~\eqref{eq:Xt-def-proj}. Under Assumption~\ref{as:gradients}, the following hold for all $t\ge 1$:
\begin{itemize}
    \item[(i)] \emph{Drift negativity:}
    \begin{equation}
    \label{eq:drift-negative-proj}
        \mathbb{E}[X_t \mid \mathcal{F}_{t-1}] \;\le\; 0.
    \end{equation}
    \item[(ii)] \emph{Conditional sub-Gaussian bound:} for all $\lambda \ge 0$,
    \begin{equation}
    \label{eq:subg-proj}
        \mathbb{E}\!\left[
            \exp\big(\lambda X_t\big)
            \,\Big|\,
            \mathcal{F}_{t-1}
        \right]
        \;\le\;
        \exp\!\big(2\lambda^2 \sigma_t^2\big).
    \end{equation}
\end{itemize}
\end{lemma}
\begin{proof}
We first derive a convenient upper bound on $X_t$. From~\eqref{eq:proj-nonexpansion} and the definition~\eqref{eq:Xt-def-proj},
\begin{align}
    X_t
    &= 2\eta_t\big(f(x_t)-f(x^\star)\big)
       - (Z_t - Z_{t+1})
       - \eta_t^2\|g_t\|^2 \notag\\
    &\le
    2\eta_t\big(f(x_t)-f(x^\star)\big)
    - \big(2\eta_t\inner{x_t - x^\star}{g_t}
        - \eta_t^2\|g_t\|^2\big)
    - \eta_t^2\|g_t\|^2 \notag\\
    &=
    2\eta_t\Big(
        f(x_t)-f(x^\star)
        - \inner{x_t - x^\star}{g_t}
    \Big).
    \label{eq:Xt-upper-first}
\end{align}
Recalling that $\xi_t := g_t - \nabla f(x_t)$, we can further decompose $\inner{x_t - x^\star}{g_t} = \inner{x_t - x^\star}{\nabla f(x_t)}+ \inner{x_t - x^\star}{\xi_t}$. Substituting this into~\eqref{eq:Xt-upper-first} yields
\begin{equation}
\label{eq:Xt-upper-split}
    X_t
    \le
    2\eta_t\Big(
        f(x_t)-f(x^\star)
        - \inner{x_t - x^\star}{\nabla f(x_t)}
    \Big)
    - 2\eta_t\inner{x_t - x^\star}{\xi_t}.
\end{equation}
We now prove assertion~(i). By convexity of $f$,
\[
    f(x_t) - f(x^\star)
    \le
    \inner{\nabla f(x_t)}{x_t - x^\star},
\]
so the first term in~\eqref{eq:Xt-upper-split} is nonpositive: $f(x_t)-f(x^\star) - \inner{x_t - x^\star}{\nabla f(x_t)} \le 0$. Taking conditional expectations in~\eqref{eq:Xt-upper-split} and using the unbiasedness in Assumption~\ref{as:gradients}(i), we obtain
\begin{align*}
    \mathbb{E}[X_t \mid \mathcal{F}_{t-1}]
    &\le
    2\eta_t\,\mathbb{E}\!\left[
        f(x_t)-f(x^\star)
        - \inner{x_t - x^\star}{\nabla f(x_t)}
        \,\Big|\,
        \mathcal{F}_{t-1}
    \right]
    - 2\eta_t\,\mathbb{E}\!\left[
        \inner{x_t - x^\star}{\xi_t}
        \,\Big|\,
        \mathcal{F}_{t-1}
    \right]\\
    &=
    2\eta_t\Big(
        f(x_t)-f(x^\star)
        - \inner{x_t - x^\star}{\nabla f(x_t)}
    \Big)
    - 2\eta_t \inner{x_t - x^\star}{
        \mathbb{E}[\xi_t \mid \mathcal{F}_{t-1}]
    } \\
    &=
    2\eta_t\Big(
        f(x_t)-f(x^\star)
        - \inner{x_t - x^\star}{\nabla f(x_t)}
    \Big)
    \;\le\; 0,
\end{align*}
which proves~\eqref{eq:drift-negative-proj}. Note that $x_t$ is $\mathcal F_{t-1}$-measurable for each $t\ge 2$, since it is a deterministic function of $(x_1,g_1,\ldots,x_{t-1},g_{t-1})$ via \eqref{eq:proj-sgd-update}.

We now prove assertion~(ii). From~\eqref{eq:Xt-upper-split} and the convexity inequality just used, we also have the pathwise bound
\begin{equation}
\label{eq:Xt-noise-upper}
    X_t
    \;\le\;
    - 2\eta_t\inner{x_t - x^\star}{\xi_t}.
\end{equation}
Fix $\lambda \ge 0$. Multiplying~\eqref{eq:Xt-noise-upper} by $\lambda$ and taking conditional expectations,
\begin{align*}
    \mathbb{E}\!\left[
        \exp\big(\lambda X_t\big)
        \,\Big|\,
        \mathcal{F}_{t-1}
    \right]
    &\le
    \mathbb{E}\!\left[
        \exp\!\big(-2\lambda\eta_t
            \inner{x_t - x^\star}{\xi_t}
        \big)
        \,\Big|\,
        \mathcal{F}_{t-1}
    \right].
\end{align*}
Applying Assumption~\ref{as:gradients}(ii) with
\[
    u := -2\eta_t(x_t - x^\star)
    \ \implies\
    \|u\|^2 = 4\eta_t^2 \|x_t - x^\star\|^2 = 4\eta_t^2 Z_t,
\]
we obtain
\begin{equation*}
    \mathbb{E}\!\left[
        \exp\big(\lambda X_t\big)
        \,\Big|\,
        \mathcal{F}_{t-1}
    \right]
    \le
    \exp\!\left(
        \frac{\lambda^2 \sigma^2 \|u\|^2}{2}
    \right)
    =
    \exp\!\big(2\lambda^2 \sigma^2 \eta_t^2 Z_t\big)
    =
    \exp\!\big(2\lambda^2 \sigma_t^2\big),
\end{equation*}
which is exactly~\eqref{eq:subg-proj}.
\end{proof}
The drift inequality \eqref{eq:drift-negative-proj} is a standard one-step Lyapunov property for projected stochastic approximation; see, e.g., \cite{nemirovski2009robust}. The relevance of the process $X_t$ emerges once we consider its partial sums
\begin{equation}
\label{eq:partial-sum}
    \bar X_t \coloneqq \sum_{s=1}^t X_s,
\end{equation}
which admit a direct algebraic link to the cumulative weighted suboptimality:
\[
    \bar X_t
    =
    2\sum_{s=1}^t \eta_s\big(f(x_s)-f(x^\star)\big)
    - (Z_1 - Z_{t+1})
    - \sum_{s=1}^t \eta_s^2\|g_s\|^2.
\]
Since $Z_{t+1}\ge 0$, it follows that
\[
    2\sum_{s=1}^t \eta_s\big(f(x_s)-f(x^\star)\big)
    \;\le\;
    \bar X_t + Z_1 + \sum_{s=1}^t \eta_s^2\|g_s\|^2.
\]
Dividing by $2S_t$ and recalling the definition of $\bar F_t$ in~\eqref{eq:weighted:subopt}, we obtain
\[
    \bar F_t
    \;\le\;
    \frac{1}{2S_t}
    \left(
        \bar X_t
        + Z_1
        + \sum_{s=1}^t \eta_s^2\|g_s\|^2
    \right).
\]
This inequality already mirrors the structure of the confidence sequence $U_t(\alpha)$ defined in~\eqref{eq:U-def}: the first term comes from an upper bound on $\bar X_t$, while the remaining two terms carry over unchanged. Thus, controlling $\bar X_t$ uniformly over time is the key step in constructing $U_t(\alpha)$.

\bigskip
\noindent\textbf{Step 2: A time-uniform concentration bound for $\bar X_t$}

Our goal is now to show that the partial sums $\bar X_t$ admit a time-uniform upper bound in terms of the effective cumulative variance proxy $\Sigma_{t,\mathrm{eff}}^2$ defined in \eqref{eq:sigma-Sigma-def}. This is where the main probabilistic machinery enters: we construct an exponential supermartingale for each fixed $\lambda$, mix over a geometric grid of $\lambda$'s, and apply Ville’s inequality to obtain a high-probability bound that holds simultaneously for all $t$.

To state the result, define, for $k=0,1,2,\ldots$,
\[
    \lambda_k := \mathrm{e}^{-k/2},
    \qquad
    w_k := \frac{6}{\pi^2 (k+1)^2}.
\]
Note that $\sum_{k= 0}^\infty w_k = 1$. We are now ready to state the time-uniform concentration result for $\bar X_t$.

\begin{lemma}[Time-uniform bound for $\bar X_t$]
\label{lem:cs-X-full}
Fix $\alpha\in(0,\tfrac{2}{\mathrm{e}})$, and let $\{X_t\}_{t\ge 0}$ be the process defined in \eqref{eq:Xt-def-proj}. For $t\ge 0$, define the mixture process
\[
    \bar E_t
    :=
    \sum_{k= 0}^\infty w_k
        \exp\!\left(\lambda_k \bar X_t - 2\lambda_k^2\,\Sigma_t^2\right),
\]
with $\Sigma_t^2$ defined in \eqref{eq:sigma-Sigma-def}. Then $\{\bar E_t\}_{t\ge 0}$ is a nonnegative supermartingale with $\bar E_0 = 1$. Moreover, on the event $\mathcal{E}_\alpha := \bigl\{\forall t\ge1 : \bar E_t < \tfrac{1}{\alpha} \bigr\}$, which satisfies $\mathbb{P}(\mathcal{E}_\alpha) \ge 1 - \alpha$, the following holds: 
\[
    \forall t\ge 1:\
    \bar X_t
    \le
    7\sqrt{\Sigma_{t,\mathrm{eff}}^2
    \Big(\log\frac{2}{\alpha}
    + \log\log(\mathrm e+\Sigma_{t,\mathrm{eff}}^2)\Big)}.
\]
\end{lemma}
\begin{proof}
For clarity, we divide the proof into six steps.

\medskip
\noindent
\textbf{Step (i): $\{\bar E_t\}_{t\ge 0}$ is a nonnegative supermartingale.}
For any fixed $\lambda\ge 0$, define
\[
    E_t(\lambda)
    :=
    \exp\!\left(\lambda \bar X_t - 2\lambda^2 \Sigma_t^2\right)
    \qquad \forall t\ge 0,
\]
with the convention $\bar X_0=0$ and $\Sigma_0^2=0$, so $E_0(\lambda)=1$. Using $\bar X_t=\bar X_{t-1}+X_t$ and $\Sigma_t^2=\Sigma_{t-1}^2+\sigma_t^2$, we compute
\begin{align*}
    \mathbb{E}[E_t(\lambda)\mid \mathcal{F}_{t-1}]
    &=
    \mathbb{E}\!\left[
        \exp\!\left(
            \lambda(\bar X_{t-1}+X_t)
            - 2\lambda^2(\Sigma_{t-1}^2+\sigma_t^2)
        \right)
        \Bigm| \mathcal{F}_{t-1}
    \right] \\
    &=
    \exp\!\left(\lambda \bar X_{t-1} - 2\lambda^2 \Sigma_{t-1}^2\right)
    \mathbb{E}\!\left[
        \exp\!\left(
            \lambda X_t - 2\lambda^2 \sigma_t^2
        \right)
        \Bigm| \mathcal{F}_{t-1}
    \right] \\
    &=
    E_{t-1}(\lambda)\,
    \mathbb{E}\!\left[
        \exp\!\left(\lambda X_t - 2\lambda^2 \sigma_t^2\right)
        \Bigm| \mathcal{F}_{t-1}
    \right].
\end{align*}
By the conditional sub-Gaussian assumption \eqref{eq:subg-proj}, the conditional expectation in the last line is at most one. Hence
\[
    \mathbb{E}[E_t(\lambda)\mid \mathcal{F}_{t-1}]
    \le E_{t-1}(\lambda),
\]
so $\{E_t(\lambda)\}_{t\ge 0}$ is a nonnegative supermartingale with $E_0(\lambda)=1$. Finally, since each $E_t(\lambda_k)$ is a nonnegative supermartingale and $(w_k)_{k\ge 0}$ is a fixed sequence of nonnegative weights summing to one, $\{\bar E_t\}_{t\ge 0}$ is also a nonnegative supermartingale with $\bar E_0 = \sum_{k\ge 0}w_k = 1$.

\bigskip
\noindent
\textbf{Step (ii): A first upper bound on $\bar X_t$.}
By Ville’s inequality applied to the nonnegative supermartingale $\bar E_t$,
\[
    \mathbb{P}\!\left(
        \sup_{t\ge 1} \bar E_t \ge \tfrac{1}{\alpha}
    \right)
    \le \alpha.
\]
Equivalently, the event
\[
    \mathcal{E}_\alpha := \Big\{\forall t\ge 1:\ \bar E_t < \tfrac{1}{\alpha}\Big\}
\]
satisfies $\mathbb{P}(\mathcal{E}_\alpha)\ge 1-\alpha$. Fix $\omega\in\mathcal{E}_\alpha$ and a time $t\ge 1$ with $\Sigma_t^2>0$. For notational simplicity, we suppress the dependence on $\omega$ in what follows. On $\mathcal{E}_\alpha$ we have
\[
    \bar E_t
    =
    \sum_{k\ge 0} w_k
    \exp\!\left(\lambda_k \bar X_t - 2\lambda_k^2\Sigma_t^2\right)
    < \frac{1}{\alpha}.
\]
Hence for every $k\ge 0$,
\[
    w_k\exp\!\left(\lambda_k \bar X_t-2\lambda_k^2\Sigma_t^2\right)
    < \frac{1}{\alpha},
\]
which implies
\[
    \lambda_k \bar X_t-2\lambda_k^2\Sigma_t^2
    < \log\frac{1}{\alpha w_k}.
\]
Rearranging this inequality yields, for each $k\ge 0$,
\[
    \bar X_t
    <
    \frac{\log(1/(\alpha w_k))}{\lambda_k}
    + 2\lambda_k\Sigma_t^2
    =: B_k(\Sigma_t^2).
\]
Noticing that $B_k(\cdot)$ is nondecreasing and $\Sigma_t^2 \le \Sigma_{t,\mathrm{eff}}^2$, we further have
\[
    \bar X_t < B_k(\Sigma_{t,\mathrm{eff}}^2)
    \qquad \forall k\ge 0.
\]
and since this holds simultaneously for all $k\ge 0$, we obtain the bound
\begin{equation}
\label{eq:inf-Bk}
    \bar X_t
    \le \inf_{k\ge 0} B_k(\Sigma_{t,\mathrm{eff}}^2).
\end{equation}
It remains to upper bound the infimum in terms of $\Sigma_{t,\mathrm{eff}}^2$ and $\alpha$.

\bigskip
\noindent
\textbf{Step (iii): Explicit form of $B_k(\Sigma_{t,\mathrm{eff}}^2)$.}
For notational simplicity, set
\[
    v := \Sigma_{t,\mathrm{eff}}^2,
    \qquad
    L_\alpha := \log\frac{2}{\alpha} \ge \log\frac{\pi^2}{6 \alpha}.
\]
Using $w_k = {6}/(\pi^2 (k+1)^2)$ and $\lambda_k = \mathrm{e}^{-k/2}$, we have
\[
    \log\frac{1}{\alpha w_k}
    = \log\frac{\pi^2}{6 \alpha} + 2\log(k+1)
    \le L_\alpha + 2\log(k+1),
\]
and $1/{\lambda_k} = \mathrm{e}^{k/2}$. Hence, for $v>0$,
\[
    B_k(v)
    :=
    \frac{\log(1/(\alpha w_k))}{\lambda_k}
    + 2\lambda_k v
    \le
    \mathrm{e}^{k/2}
    \big(L_\alpha + 2\log(k+1)\big)
    + 2 \mathrm{e}^{-k/2} v.
\]

\bigskip
\noindent
\textbf{Step (iv): Choose $k$ as a function of $v$.}
Define
\[
    L_v := \log\log(\mathrm e+v) > 0.
\]
We will pick an integer $k=k(v)$ such that:
\begin{enumerate}
    \item $\mathrm{e}^{k(v)/2}$ is within a constant factor of
    \[
        z_v := \sqrt{\frac{v}{L_\alpha + L_v}},
    \]
    \item $\log(k(v)+1)$ is bounded by a constant multiple of $L_v$.
\end{enumerate}
One convenient choice is
\[
    k(v) := \left\lfloor \log\frac{v}{L_\alpha + L_v} \right\rfloor
    \qquad\text{for } v>0,
\]
which immediately implies that $\mathrm{e}^{k(v)/2} \le z_v \le \mathrm{e}^{(k(v)+1)/2}$. Since $v = \Sigma_{t,\mathrm{eff}}^2 \ge 2(\log \tfrac{2}{\alpha}+1)$ by definition, Lemma~\ref{lem:explicit-threshold} guarantees that $k(v)\ge 0$. For the lemma, we note that $L_\alpha\ge 1$ holds since $\alpha \leq \tfrac{2}{\mathrm{e}}$. Therefore, for this choice, we have that
\begin{equation}\label{eq:zv-bounds}
    \frac{1}{\sqrt{\mathrm{e}}}\, z_v \le \mathrm{e}^{k(v)/2} \le z_v,
    \qquad
    \log(k(v)+1) \le L_v + 1,
\end{equation}
where the second inequality in \eqref{eq:zv-bounds} can be justified as follows:
\[
    \log\!\big(k(v)+1\big)
    \;\le\;
    \log\!\left(
        \log\frac{v}{L_\alpha + L_v} + 1
    \right)
    \;\le\;
    \log(\log v + 1)
    \;\le\;
    \log(\mathrm e \log v)
    \;\le\;
    1 + \log\log(\mathrm e + v).
\]
Here, the first inequality follows from the definition of $k(v)$. The second inequality uses the fact that $L_\alpha + L_v > 1$. The third inequality follows from $\log v + 1 \le \mathrm e \log v$ for all $v \ge \mathrm e$. Finally, the last inequality holds because $\log(\mathrm e + v) \ge \log v$. For this choice of $k = k(v)$, define
\[
    D_v := L_\alpha + 2\log\big(k(v)+1\big).
\]
Using \eqref{eq:zv-bounds}, we have $2\log(k(v)+1) \le 2L_v + 2$, and therefore
\[
    D_v
    \le
    L_\alpha + 2 + 2L_v
    \le
    3\big(L_\alpha + L_v\big),
\]
where the last inequality follows from the fact that $L_\alpha \ge 1$.

\bigskip
\noindent
\textbf{Step (v): Bounding $B_{k(v)}(v)$.}
Using \eqref{eq:zv-bounds} again, we have
\[
    \mathrm{e}^{k(v)/2}
    \le z_v
    = \sqrt{\frac{v}{L_\alpha + L_v}},
    \qquad
    \mathrm{e}^{-k(v)/2}
    \le \sqrt{\mathrm{e}}
        \sqrt{\frac{L_\alpha + L_v}{v}}.
\]
Therefore,
\[
    {\mathrm{e}^{k(v)/2}} D_v
    \le
    3
    \sqrt{v\big(L_\alpha + L_v\big)},
\]
and
\[
    2 \mathrm{e}^{-k(v)/2} v
    \le
    2 \sqrt{\mathrm{e}}
    \sqrt{\frac{L_\alpha + L_v}{v}}\, v 
    =
    2\sqrt{\mathrm{e}}
    \sqrt{v\big(L_\alpha + L_v\big)}.
\]
Summing these two bounds, we obtain
\[
    B_{k(v)}(v)
    \le
    \big(3+2\sqrt{\mathrm e}\big)\sqrt{v\big(L_\alpha + L_v\big)}
    \le
    7\sqrt{v\big(L_\alpha + L_v\big)}.
\]
Finally, since $\inf_k B_k(v) \le B_{k(v)}(v)$ for each fixed $v>0$, we conclude that
for all $v>0$,
\[
    \inf_{k\ge 0} B_k(v)
    \le
    7\sqrt{v\big(L_\alpha + \log\log(\mathrm e+v)\big)}
    =
    7\sqrt{v\big(\log\tfrac{2}{\alpha} + \log\log(\mathrm e+v)\big)}.
\]

\bigskip
\noindent
\textbf{Step (vi): Conclude the time-uniform bound.}
Combining this with \eqref{eq:inf-Bk}, we obtain, on the event
$\mathcal{E}_\alpha$,
\[
    \bar X_t
    \le
    7\sqrt{\Sigma_{t,\mathrm{eff}}^2
        \Big(\log\tfrac{2}{\alpha} +\log\log(\mathrm e+\Sigma_{t,\mathrm{eff}}^2)\Big)}
    \qquad\forall t\ge 1,
\]
which gives the stated bound.
\end{proof}

This lemma is the probabilistic engine of the paper: whenever one can certify a conditional sub-Gaussian inequality of the form~\eqref{eq:subg-proj}, one obtains a time-uniform concentration inequality for the partial sums, with a variance scale $\Sigma_t^2$ and only a mild iterated-logarithm correction. In particular, $\bar X_t$ remains controlled simultaneously over all $t\ge 1$ with probability at least $1-\alpha$.

\bigskip
\noindent
\textbf{Step 3: From martingale sums to cumulative suboptimality.}

Using the definition of $X_t$, we now rewrite $\bar X_t$ in terms of the cumulative weighted suboptimality, the distance terms $Z_t$, and the squared gradient terms. As explained at the end of Step~1, a simple telescoping argument yields the desired upper bound on the numerator of $\bar F_t$.

\begin{lemma}[From $\bar X_t$ to cumulative suboptimality]
\label{lem:suboptimality-bound-proj}
For any $\alpha\in(0,\tfrac{2}{\mathrm{e}})$, with probability at least $1-\alpha$, the following holds simultaneously for all $t\ge1$:
\begin{equation}\label{eq:sum-eta-gap-proj}
    \sum_{s=1}^t \eta_s\big(f(x_s)-f(x^\star)\big)
    \;\le\;
    \frac{1}{2}
    \Bigg(
        7\,\sqrt{
            \Sigma_{t,\mathrm{eff}}^2\Big(\log\tfrac{2}{\alpha}
            + \log\log(\mathrm e + \Sigma_{t,\mathrm{eff}}^2)\Big)}
        + Z_1
        + \sum_{s=1}^t \eta_s^2\|g_s\|^2
    \Bigg).
\end{equation}
\end{lemma}
\begin{proof}
From Lemma~\ref{lem:cs-X-full} we know that for any $\alpha\in(0,\tfrac{2}{\mathrm{e}})$, there exists an event $\mathcal{E}_\alpha$ with $\mathbb{P}(\mathcal{E}_\alpha)\ge 1-\alpha$ such that, on $\mathcal{E}_\alpha$,
\begin{equation}
\label{eq:Xt-sum-bound}
    \forall t\ge1:\qquad
    \sum_{s=1}^t X_s
    \le
    7\sqrt{\Sigma_{t,\mathrm{eff}}^2\Big(\log\tfrac{2}{\alpha}
        + \log\log(\mathrm e + \Sigma_{t,\mathrm{eff}}^2)\Big)}.
\end{equation}
Next, we express $\sum_{s=1}^t X_s$ in terms of the cumulative suboptimality. Summing the definition \eqref{eq:Xt-def-proj} from $s=1$ to $t$ gives
\begin{align*}
    \sum_{s=1}^t X_s
    &=
    2\sum_{s=1}^t \eta_s\big(f(x_s)-f(x^\star)\big)
    - \sum_{s=1}^t (Z_s - Z_{s+1})
    - \sum_{s=1}^t \eta_s^2\|g_s\|^2
    \\&=2\sum_{s=1}^t \eta_s\big(f(x_s)-f(x^\star)\big)
    - (Z_1 - Z_{t+1})
    - \sum_{s=1}^t \eta_s^2\|g_s\|^2.
\end{align*}
Rearranging this identity, we obtain
\begin{equation}\label{eq:sum-gap-vs-X}
    2\sum_{s=1}^t \eta_s\big(f(x_s)-f(x^\star)\big)
    =
    \sum_{s=1}^t X_s
    + (Z_1 - Z_{t+1})
    + \sum_{s=1}^t \eta_s^2\|g_s\|^2.
\end{equation}
On the event $\mathcal{E}_\alpha$, we may bound $\sum_{s=1}^t X_s$ using \eqref{eq:Xt-sum-bound}. Moreover, $Z_{t+1} = \|x_{t+1}-x^\star\|^2\ge0$, so we can drop it. Plugging these bounds into \eqref{eq:sum-gap-vs-X}, we find that on $\mathcal{E}_\alpha$, for all $t\ge1$,
\[
    2\sum_{s=1}^t \eta_s\big(f(x_s)-f(x^\star)\big)
    \le
    7\sqrt{
        \Sigma_{t,\mathrm{eff}}^2\Big(\log\tfrac{2}{\alpha}
        + \log\log(\mathrm e + \Sigma_{t,\mathrm{eff}}^2)\Big)}
    + Z_1
    + \sum_{s=1}^t \eta_s^2\|g_s\|^2.
\]
Dividing both sides by $2$ yields \eqref{eq:sum-eta-gap-proj}, and the bound holds simultaneously for all $t\ge1$ on the event $\mathcal{E}_\alpha$, which has probability at least $1-\alpha$.
\end{proof}

This lemma is the bridge from the martingale world back to optimization: it translates the time-uniform control of the process $\bar X_t$ into a time-uniform upper bound on the cumulative weighted suboptimality $\sum_{s=1}^t \eta_s(f(x_s)-f(x^\star))$.  Finally, dividing both sides by $S_t$ yields a time-uniform bound for the weighted average suboptimality $\bar F_t$, completing the proof of Theorem~\ref{thm:anytime-suboptimality}.

%--------------------------------------------------------------------------------

\subsection{Clarifying remarks and intuition}
\label{subsec:remarks-CS}

The proof of Theorem~\ref{thm:anytime-suboptimality} relies on several ideas from time-uniform sequential analysis that may be unfamiliar to readers with a background in stochastic optimization. In this subsection we collect two clarifying remarks intended to illuminate the construction. First, we explain the role of the mixture supermartingale $\bar E_t$ and why a single-parameter exponential martingale cannot yield a time-uniform bound without sacrificing optional stopping. Second, we provide a heuristic ``continuous optimization'' argument that motivates the geometric grid over $\lambda_k$ and the emergence of an iterated-logarithm term in the boundary. 

\begin{remark}[Why the mixture $\bar E_t$ is necessary]
To understand the need for mixtures in Lemma~\ref{lem:cs-X-full}, it is instructive to redo the key step with a \emph{single} exponential supermartingale. For any fixed $\lambda>0$, define
\[
    E_t(\lambda)
    :=
    \exp\!\big(\lambda \bar X_t - 2\lambda^2 \Sigma_t^2\big),
\]
as in Step~(i) of Lemma~\ref{lem:cs-X-full}.  By the conditional sub-Gaussian property \eqref{eq:subg-proj}, $\{E_t(\lambda)\}_{t\ge 0}$ is a nonnegative supermartingale with $E_0(\lambda)=1$. Ville's inequality therefore gives
\begin{equation}\label{eq:single-ville-remark}
    \mathbb{P}\!\left(
        \sup_{t\ge 1} E_t(\lambda) \,\ge\, \frac{1}{\alpha}
    \right)
    \,\le\, \alpha.
\end{equation}
Equivalently, with probability at least $1-\alpha$,
\[
    \forall t\ge 1:\
    E_t(\lambda)
    <
    \frac{1}{\alpha}
    \ \iff\
    \lambda \bar X_t - 2\lambda^2 \Sigma_t^2
    <
    \log\frac{1}{\alpha}.
\]
Rearranging yields the deterministic upper bound
\begin{equation}\label{eq:barX-single-lambda}
    \bar X_t
    <
    \frac{\log(1/\alpha)}{\lambda}
    +
    2\lambda \Sigma_t^2
    \le
    \frac{\log(1/\alpha)}{\lambda}
    +
    2\lambda \Sigma_{t,\mathrm{eff}}^2
    \qquad \forall t\ge 1.
\end{equation}
Now fix a time $t$ and, for that $t$, treat the right-hand side as a function of $\lambda$:
\begin{equation}\label{eq:B-lambda-t-def}
    B(\lambda,t)
    :=
    \frac{\log(1/\alpha)}{\lambda}
    + 2\lambda \Sigma_{t,\mathrm{eff}}^2,
    \qquad \lambda>0.
\end{equation}
This function is minimized explicitly at
\begin{equation}\label{eq:lambda-star}
    \lambda^\star(t)
    =
    \sqrt{\frac{\log(1/\alpha)}{2\Sigma_{t,\mathrm{eff}}^2}},
\end{equation}
with minimum value
\begin{equation}
\label{eq:B-lambda-star}
    B\big(\lambda^\star(t),t\big)
    =
    2\sqrt{2}\sqrt{\Sigma_{t,\mathrm{eff}}^2 \log\frac{1}{\alpha}}
    \le
    2\sqrt{2}\sqrt{\Sigma_{t,\mathrm{eff}}^2 \log\frac{2}{\alpha}}.
\end{equation}
Thus, for a \emph{fixed} time~$t$, a single exponential supermartingale produces an upper bound with the scaling $\sqrt{\Sigma_{t,\mathrm{eff}}^2\log(2/\alpha)}$. This matches the leading term that appears in Theorem~\ref{thm:anytime-suboptimality}, but \emph{without} the additional iterated-logarithm factor that arises when enforcing time-uniform validity.

The problem arises from the fact that $\lambda^\star(t)$ in \eqref{eq:lambda-star} depends on $\Sigma_{t,\mathrm{eff}}^2$ and is therefore $\mathcal{F}_t$-measurable: it is chosen \emph{after observing the data up to time $t$}. Substituting this data-dependent $\lambda^\star(t)$ back into $E_t(\lambda)$ destroys the supermartingale structure, so Ville's inequality cannot be applied time-uniformly. With a single exponential process one faces a choice: either fix $\lambda$ in advance (yielding time-uniform but suboptimal bounds for many $t$), or optimize $\lambda$ at each $t$ (yielding sharp fixed-time bounds but no optional-stopping guarantee).

The mixture construction in Lemma~\ref{lem:cs-X-full} reconciles these goals by fixing a countable family of parameters \emph{a priori} and combining the corresponding exponential supermartingales. This ``method of mixtures'' (also called pseudo-maximization) goes back to Robbins and Siegmund's classical work on sequential boundary crossing and is developed systematically in the theory of self-normalized processes (see \cite{de2009self}). Modern confidence-sequence constructions build on this principle in various forms, including explicit mixture and stitching frameworks such as the sub-$\psi$ approach in~\cite{howard2021time}, as well as betting- and e-process-based methods that aggregate families of nonnegative supermartingales; see \cite{waudby2024estimating}.   \hfill $\clubsuit$
\end{remark}

\begin{remark}[Intuition for the geometric grid $\{\lambda_k\}_{k\ge 0}$]
The mixture construction in Lemma~\ref{lem:cs-X-full} hinges on choosing a \emph{geometric} grid of tuning parameters $\lambda_k = \mathrm e^{-k/2}$. Here we provide intuition for this choice by examining how the envelope $B_k(\Sigma_{t,\mathrm{eff}}^2)$ is minimized, and how the geometry of the grid matches the natural scale of the problem. Let $v \coloneqq \Sigma_{t,\mathrm{eff}}^2$ and recall the upper envelope
\[
    B_k(v)
    :=
    \frac{\log(1/(\alpha w_k))}{\lambda_k}
    + 2\lambda_k v,
    \qquad
    \lambda_k = \mathrm e^{-k/2},
    \quad
    w_k = \frac{6}{\pi^2(k+1)^2}.
\]
Momentarily regard $k$ as a continuous variable and write $z=\mathrm e^{k/2}>0$. Then
\[
    B_k(v)
    =
    D(z)\,z
    + 2\,\frac{v}{z},
\]
where
\[
    D(z)
    :=
    \log\frac{\pi^2}{6\alpha}
    + 2\log(k+1)
    =
    \log\frac{\pi^2}{6\alpha}
    + 2\log\!\big(1 + 2\log z\big).
\]
The key observation is that $D(z)$ varies \emph{very slowly} in $z$: once $z$ is chosen near its optimal scale, $D(z)$ grows only like
$\log\log(\mathrm e+z)$, a negligible rate compared to any power of $z$. As a heuristic, we therefore \emph{treat $D(z)$ as approximately constant} and consider the simplified objective
\[
    B(z,v)
    :=
    D\,z
    + 2\,\frac{v}{z},
\]
where $D>0$ is held fixed. The function $B(z,v)$ is strictly convex in $z>0$, and its minimum is attained at
\[
    z^\star
    =
    \sqrt{\frac{2v}{D}}
    \ \implies\
    \lambda^\star
    =
    \frac{1}{z^\star}
    =
    \sqrt{\frac{D}{2v}}.
\]
Substituting $z^\star$ back into $B(z,v)$ yields
\[
    \min_{z > 0} B(z,v)
    =
    2\sqrt{2Dv}.
\]
Heuristically, since $D(z) \approx \log\tfrac{\pi^2}{6\alpha} + \log\log(\mathrm e+z)$, which at the relevant scale is equivalent up to
constants to $\log\tfrac{2}{\alpha} + \log\log(\mathrm e+v)$, this suggests the leading-order behavior
\[
    \inf_{k\ge 0} B_k(v)
    \;\lesssim\;
    \sqrt{v\Big(\log\tfrac{2}{\alpha} + \log\log(\mathrm e+v)\Big)}.
\]
The slowly varying term $\log(k+1)$ introduces an iterated-logarithm correction, where the $\log\log(\mathrm e+v)$ factor arises precisely from the slow growth of $\log(k+1)$ on this geometric grid. In particular, for a fixed time $t$, the
ideal tuning satisfies $\lambda^\star(t)\asymp \Sigma_{t,\mathrm{eff}}^{-1}$. Since $\Sigma_{t,\mathrm{eff}}^2$ is unknown in
advance and varies with the trajectory, a fixed geometric grid of the form $\lambda_k=\mathrm e^{-k/2}$ ensures that for every $t$ there exists a grid point $\lambda_k$ within a constant factor of the ideal choice $\lambda^\star(t)$.

In short, the grid is chosen so that: (i) it spans the relevant scales of $\Sigma_{t,\mathrm{eff}}^2$ via exponentially spaced search points, and (ii) at least one $\lambda_k$ is guaranteed to be within a constant factor of the ideal minimizer, thereby achieving near-optimal behavior across all times $t$. This ``scale-matching'' intuition for geometric discretizations is common in mixture
and stitching arguments for martingale concentration; see, for example, the heuristic discussion surrounding Theorem~1 in~\cite{howard2021time} and the broader treatment of pseudo-maximization via mixtures in self-normalized processes~\cite{de2009self}. \hfill $\clubsuit$
\end{remark}

%--------------------------------------------------------------------------------

\subsection{Proof of Theorem~\ref{thm:anytime-stationarity:nonconvex}}
\label{subsec:proof:thm:anytime-stationarity:nonconvex}

The proof of Theorem~\ref{thm:anytime-stationarity:nonconvex} follows the same overall structure as in the convex case, but the technical arguments differ in substantive ways due to the shift from optimality to stationarity and from convexity to smoothness assumptions. As the time-uniform concentration machinery is already developed in the
convex setting, we adopt a more direct presentation here, emphasizing the distinct smoothness-based decomposition and the resulting martingale structure that drive the nonconvex analysis. 

In contrast to the convex analysis, where the argument centers on the process $\bar X_t$, the nonconvex proof is organized around the cumulative martingale term
\begin{equation}
\label{eq:MT-def}
    \bar Y_t
    :=
    -\sum_{s=1}^t \eta_s \langle \nabla f(x_s), \xi_s\rangle,
    \qquad
    \text{with } \xi_s := g_s - \nabla f(x_s),
\end{equation}
which captures the aggregated effect of stochastic gradient noise along the SGD trajectory. Using a smoothness-based descent inequality, we rewrite the weighted stationarity measure in terms of a descent contribution and the martingale term $\bar Y_t$. We then establish conditional sub-Gaussian control of $\bar Y_t$ and construct an associated exponential supermartingale. A time-uniform concentration bound for $\bar Y_t$, combined with a telescoping argument and normalization by $S_t$, yields an anytime-valid upper confidence sequence for the weighted stationarity measure $\bar G_t$.

We begin by deriving a decomposition of the weighted stationarity measure along the SGD trajectory. In contrast to the convex case, where progress is tracked through the distance process $Z_t$, the nonconvex analysis relies on the smoothness descent inequality. The following lemma expresses the cumulative weighted squared gradient norm as the sum of a smoothness-based descent term and a martingale error term induced by stochastic gradient noise.

\begin{lemma}[Weighted stationarity decomposition]
\label{lem:descent-decomposition}
Under Assumption~\ref{as:smoothness}, let $\{x_s\}$ be generated by the SGD recursion
\eqref{eq:proj-sgd-update} with $\mathcal{X}=\mathbb{R}^d$. Then, for every $t\ge 1$,
\begin{equation}
\label{eq:descent-decomposition}
    \sum_{s=1}^t \eta_s \|\nabla f(x_s)\|^2
    \;\le\;
    f(x_1)
    +\frac{L}{2}\sum_{s=1}^t \eta_s^2 \|g_s\|^2
    + \bar Y_t,
\end{equation}
where $\bar Y_t$ is defined in \eqref{eq:MT-def}.
\end{lemma}
\begin{proof}
Because $f$ is $L$-smooth, for every $s\ge 1$ we have the standard upper bound
\begin{equation}
\label{eq:smoothness-upper}
    f(x_{s+1})
    \le
    f(x_s)
    +\langle \nabla f(x_s),x_{s+1}-x_s\rangle
    +\frac{L}{2}\|x_{s+1}-x_s\|^2.
\end{equation}
Substituting the SGD update $x_{s+1}-x_s=-\eta_s g_s$ and rearranging yields
\begin{equation}
\label{eq:innerprod-lower}
    \eta_s\langle \nabla f(x_s),g_s\rangle
    \le
    f(x_s)-f(x_{s+1})
    +\frac{L}{2}\eta_s^2\|g_s\|^2.
\end{equation}
Writing $g_s=\nabla f(x_s)+\xi_s$, we have $\langle \nabla f(x_s),g_s\rangle = \|\nabla f(x_s)\|^2+\langle \nabla f(x_s),\xi_s\rangle$. Substituting this decomposition into \eqref{eq:innerprod-lower} yields
\begin{equation}
\label{eq:key-reduction}
    \eta_s\|\nabla f(x_s)\|^2
    \le
    f(x_s)-f(x_{s+1})
    +\frac{L}{2}\eta_s^2\|g_s\|^2
    -\eta_s\langle \nabla f(x_s),\xi_s\rangle.
\end{equation}
Summing \eqref{eq:key-reduction} over $s=1,\dots,t$ and telescoping the function values gives
\[
    \sum_{s=1}^t \eta_s\|\nabla f(x_s)\|^2
    \le
    f(x_1)-f(x_{t+1})
    +\frac{L}{2}\sum_{s=1}^t \eta_s^2\|g_s\|^2
    -\sum_{s=1}^t \eta_s\langle \nabla f(x_s),\xi_s\rangle.
\]
Since $f$ is nonnegative by construction, $f(x_{t+1})\ge 0$, and the result follows with
$\bar Y_t$ defined in \eqref{eq:MT-def}.
\end{proof}

Lemma~\ref{lem:descent-decomposition} follows the standard smoothness descent inequality used in nonconvex SGD analyses; see, e.g., \cite{ghadimi2013stochastic}. It reduces the analysis of the weighted stationarity measure to controlling the cumulative martingale term $\bar Y_t$. The next lemma constructs a nonnegative exponential supermartingale associated with $\bar Y_t$ under the conditional sub-Gaussian noise assumption. This construction serves as the key probabilistic ingredient for obtaining time-uniform control of the stochastic error accumulated along the SGD trajectory.

\begin{lemma}[Defining an exponential supermartingale]
\label{lem:exp-supermartingale}
Suppose Assumption~\ref{as:gradients} holds. Let $\bar Y_t$ be defined as in \eqref{eq:MT-def}, and let $\{\nu_t^2\}$ and $\{\Gamma_t^2\}$ be defined as in \eqref{eq:sigma-Sigma-def:nonconvex}. Then, for every $\lambda\ge 0$, the process
\[
    E_t(\lambda)
    :=
    \exp\!\left(
        \lambda \bar Y_t-\frac{\lambda^2}{2}\Gamma_t^2
    \right)\qquad\text{for all } t\ge 1,
\]
(with $\bar Y_0=\Gamma_0^2=0$) is a nonnegative supermartingale with respect to $\{\mathcal{F}_t\}$.
\end{lemma}
\begin{proof}
For notational convenience, define
\[
    Y_t := -\eta_t\langle \nabla f(x_t),\xi_t\rangle,
    \qquad\text{so that}\qquad
    \bar Y_t=\sum_{s=1}^t Y_s.
\]
Fix $t\ge 1$ and $\lambda\ge 0$. Since $\eta_t$ and $\nabla f(x_t)$ are $\mathcal{F}_{t-1}$-measurable, Assumption~\ref{as:gradients}(ii) applied with $u:=-\eta_t\nabla f(x_t)$ yields
\begin{align}
\label{eq:mgf-Yt}
    \mathbb{E}\!\left[\exp(\lambda Y_t)\,\big|\,\mathcal{F}_{t-1}\right]
    &=
    \mathbb{E}\!\left[
        \exp\!\big(\lambda\langle -\eta_t\nabla f(x_t),\xi_t\rangle\big)
        \,\Big|\,\mathcal{F}_{t-1}
    \right]\notag
    \\&\le
    \exp\!\left(
        \frac{\lambda^2\sigma^2\eta_t^2\|\nabla f(x_t)\|^2}{2}
    \right)
    =
    \exp\!\left(\frac{\lambda^2\nu_t^2}{2}\right),
\end{align}
where the last equality uses the definition of $\nu_t^2$ in \eqref{eq:sigma-Sigma-def:nonconvex}. Using $\bar Y_t=\bar Y_{t-1}+Y_t$ and $\Gamma_t^2=\Gamma_{t-1}^2+\nu_t^2$, we obtain
\begin{align*}
    \mathbb{E}\!\left[E_t(\lambda)\,\big|\,\mathcal{F}_{t-1}\right]
    &=
    \exp\!\left(
        \lambda \bar Y_{t-1}-\frac{\lambda^2}{2}\Gamma_{t-1}^2
    \right)
    \mathbb{E}\!\left[
        \exp\!\left(\lambda Y_t-\frac{\lambda^2}{2}\nu_t^2\right)
        \Bigm|\,\mathcal{F}_{t-1}
    \right]\\
    &=
    E_{t-1}(\lambda)\,
    \mathbb{E}\!\left[
        \exp\!\left(\lambda Y_t-\frac{\lambda^2}{2}\nu_t^2\right)
        \Bigm|\,\mathcal{F}_{t-1}
    \right]\\
    &\le
    E_{t-1}(\lambda),
\end{align*}
where the inequality follows from \eqref{eq:mgf-Yt}. Therefore, for each fixed $\lambda\ge 0$, $\{E_t(\lambda)\}_{t\ge 0}$ is a nonnegative supermartingale. Finally, $E_0(\lambda)=1$ since $\bar Y_0=\Gamma_0^2=0$.
\end{proof}

To obtain an explicit time-uniform bound from the exponential supermartingale in Lemma~\ref{lem:exp-supermartingale}, we apply the same mixture construction and Ville’s inequality used in the convex case. This yields a single high-probability event on which the cumulative noise term $\bar Y_t$ is simultaneously controlled for all times.

\begin{lemma}[Time-uniform bound for $\bar Y_t$]
\label{lem:mixture-bound}
Fix $\alpha\in(0,\tfrac{2}{\mathrm{e}})$. Under the conditions of Lemma~\ref{lem:exp-supermartingale} define
\[
    w_k := \frac{6}{\pi^2(k+1)^2},
    \qquad
    \lambda_k := \mathrm e^{-k/2}
    \qquad\text{for all } k\ge 0,
\]
and let
\[
    \bar{E}_t
    :=
    \sum_{k=0}^\infty w_k\,E_t(\lambda_k)
    \qquad\text{for all } t\ge 0,
\]
where $E_t(\lambda)$ is defined in Lemma~\ref{lem:exp-supermartingale}. Then
$\{\bar{E}_t\}_{t\ge 0}$ is a nonnegative supermartingale with $\bar{E}_0=1$, and the event
\[
    \mathcal{E}_\alpha
    :=
    \left\{\forall t\ge 1:\ \bar{E}_t < \frac{1}{\alpha}\right\}
\]
satisfies $\mathbb{P}(\mathcal{E}_\alpha)\ge 1-\alpha$. Moreover, on $\mathcal{E}_\alpha$, the following holds:
\begin{equation}
\label{eq:Mt-uniform-bound}
    \forall t\ge 1:\ \bar Y_t
    \le
    4\,\sqrt{\Gamma_{t,\mathrm{eff}}^2
    \Big(\log\tfrac{2}{\alpha}
    + \log\log(\mathrm e+\Gamma_{t,\mathrm{eff}}^2)\Big)},
\end{equation}
where $\Gamma_{t,\mathrm{eff}}^2$ is defined in \eqref{eq:sigma-Sigma-def:nonconvex}.
\end{lemma}
\begin{proof}[Proof sketch]
The argument mirrors Lemma~\ref{lem:cs-X-full}. The only change is the quadratic term in the exponential: here it is $(\lambda^2/2)\Gamma_t^2$ rather than $2\lambda^2\Sigma_t^2$, which replaces the envelope term $2\lambda_k v$ by $(\lambda_k/2)v$ in the analogue of $B_k(v)$. Tracking this change through Steps~(ii)--(v) yields the stated constant $4$.
\end{proof}

We are now ready to reconnect the preceding martingale control to the weighted average stationarity measure $\bar G_t$. On the high-probability event $\mathcal E_\alpha$ from Lemma~\ref{lem:mixture-bound}, the time-uniform bound \eqref{eq:Mt-uniform-bound} holds for all $t\ge 1$. Combining this with the smoothness-based decomposition from Lemma~\ref{lem:descent-decomposition} yields, for every $t\ge 1$,
\[
    \sum_{s=1}^t \eta_s \|\nabla f(x_s)\|^2
    \le
    f(x_1)
    +\frac{L}{2}\sum_{s=1}^t \eta_s^2 \|g_s\|^2
    +4\,\sqrt{\Gamma_{t,\mathrm{eff}}^2
    \Big(\log\tfrac{2}{\alpha}
    + \log\log(\mathrm e+\Gamma_{t,\mathrm{eff}}^2)\Big)}.
\]
Dividing both sides by $S_t=\sum_{s=1}^t \eta_s$ and recalling the definition $\bar G_t := \frac{1}{S_t}\sum_{s=1}^t \eta_s\|\nabla f(x_s)\|^2$ from~\eqref{eq:def-Gbar} gives the claimed anytime-valid upper confidence sequence. Since $\mathbb{P}(\mathcal{E}_\alpha)\ge 1-\alpha$, the bound holds simultaneously for all $t\ge 1$ with probability at least $1-\alpha$, completing the proof.

%--------------------------------------------------------------------------------
%--------------------------------------------------------------------------------
%--------------------------------------------------------------------------------

\section*{Acknowledgements}

Liviu Aolaritei acknowledges support from the Swiss National Science Foundation through the Postdoc.Mobility Fellowship (grant agreement P500PT\_222215). Funded in part by the European Union (ERC-2022-SYG-OCEAN-101071601). Views and opinions expressed are however those of the author(s) only and do not necessarily reflect those of the European Union or the European Research Council Executive Agency. Neither the European Union nor the granting authority can be held responsible for them.

\bibliographystyle{abbrvnat} 
\bibliography{bibfile.bib}

@article{howard2020time,
  title={Time-uniform {C}hernoff bounds via nonnegative supermartingales},
  author={Howard, Steven R and Ramdas, Aaditya and McAuliffe, Jon and Sekhon, Jasjeet},
  journal={Probability Surveys},
  volume={17},
  pages={257--317},
  year={2020}
}

@incollection{wald1992sequential,
  title={Sequential tests of statistical hypotheses},
  author={Wald, Abraham},
  booktitle={Breakthroughs in statistics: Foundations and basic theory},
  pages={256--298},
  year={1992},
  publisher={Springer}
}

@article{waudby2025universal,
  title={Universal log-optimality for general classes of e-processes and sequential hypothesis tests},
  author={Waudby-Smith, Ian and Sandoval, Ricardo and Jordan, Michael I},
  journal={arXiv preprint arXiv:2504.02818},
  year={2025}
}

@article{chugg2023auditing,
  title={Auditing fairness by betting},
  author={Chugg, Ben and Cortes-Gomez, Santiago and Wilder, Bryan and Ramdas, Aaditya},
  journal={Advances in Neural Information Processing Systems},
  volume={36},
  pages={6070--6091},
  year={2023}
}

@article{chen2025optimistic,
  title={Optimistic interior point methods for sequential hypothesis testing by betting},
  author={Chen, Can and Wang, Jun-Kun},
  journal={arXiv preprint arXiv:2502.07774},
  year={2025}
}

@article{shekhar2023nonparametric,
  title={Nonparametric two-sample testing by betting},
  author={Shekhar, Shubhanshu and Ramdas, Aaditya},
  journal={IEEE Transactions on Information Theory},
  volume={70},
  number={2},
  pages={1178--1203},
  year={2023},
  publisher={IEEE}
}

@article{howard2021time,
  title={Time-uniform, nonparametric, nonasymptotic confidence sequences},
  author={Howard, Steven R and Ramdas, Aaditya and McAuliffe, Jon and Sekhon, Jasjeet},
  journal={The Annals of Statistics},
  volume={49},
  number={2},
  pages={1055--1080},
  year={2021},
  publisher={JSTOR}
}

@article{waudby2024estimating,
  title={Estimating means of bounded random variables by betting},
  author={Waudby-Smith, Ian and Ramdas, Aaditya},
  journal={Journal of the Royal Statistical Society Series B: Statistical Methodology},
  volume={86},
  number={1},
  pages={1--27},
  year={2024},
  publisher={Oxford University Press US}
}

@article{nemirovski2009robust,
  title={Robust stochastic approximation approach to stochastic programming},
  author={Nemirovski, Arkadi and Juditsky, Anatoli and Lan, Guanghui and Shapiro, Alexander},
  journal={SIAM Journal on Optimization},
  volume={19},
  number={4},
  pages={1574--1609},
  year={2009},
  publisher={SIAM}
}

@article{moulines2011non,
  title={Non-asymptotic analysis of stochastic approximation algorithms for machine learning},
  author={Moulines, Eric and Bach, Francis},
  journal={Advances in Neural Information Processing Systems},
  volume={24},
  year={2011}
}

@article{dieuleveut2016nonparametric,
  author  = {Aymeric Dieuleveut and Francis Bach},
  title   = {Nonparametric Stochastic Approximation with Large Step Sizes},
  journal = {The Annals of Statistics},
  volume  = {44},
  number  = {4},
  pages   = {1363--1399},
  year    = {2016}
}

@book{de2009self,
  title={Self-normalized Processes: Limit Theory and Statistical Applications},
  author={De la Pena, Victor H and Lai, Tze Leung and Shao, Qi-Man},
  year={2009},
  publisher={Springer}
}

@article{robbins1951stochastic,
  title={A stochastic approximation method},
  author={Robbins, Herbert and Monro, Sutton},
  journal={The Annals of Mathematical Statistics},
  pages={400--407},
  year={1951},
  publisher={JSTOR}
}

@book{kushner2003stochastic,
  title={Stochastic Approximation and Recursive Algorithms and Applications},
  author={Kushner, Harold J and Yin, G George},
  year={2003},
  publisher={Springer}
}

@article{polyak1992acceleration,
  title={Acceleration of stochastic approximation by averaging},
  author={Polyak, Boris T and Juditsky, Anatoli B},
  journal={SIAM Journal on Control and Optimization},
  volume={30},
  number={4},
  pages={838--855},
  year={1992},
  publisher={SIAM}
}

@article{duchi2021asymptotic,
  title={Asymptotic optimality in stochastic optimization},
  author={Duchi, John C. and Ruan, Feng},
  journal={The Annals of Statistics},
  volume={49},
  number={1},
  pages={21--48},
  year={2021},
  publisher={Institute of Mathematical Statistics}
}

@article{davis2024asymptotic,
  title={Asymptotic normality and optimality in nonsmooth stochastic approximation},
  author={Davis, Damek and Drusvyatskiy, Dmitriy and Jiang, Liwei},
  journal={The Annals of Statistics},
  volume={52},
  number={4},
  pages={1485--1508},
  year={2024},
  publisher={Institute of Mathematical Statistics}
}

@article{lan2012optimal,
  title={An optimal method for stochastic composite optimization},
  author={Lan, Guanghui},
  journal={Mathematical Programming},
  volume={133},
  number={1},
  pages={365--397},
  year={2012},
  publisher={Springer}
}

@article{bach2013non,
  title={Non-strongly-convex smooth stochastic approximation with convergence rate O (1/n)},
  author={Bach, Francis and Moulines, Eric},
  journal={Advances in Neural Information Processing Systems},
  volume={26},
  year={2013}
}

@inproceedings{rakhlin2012making,
  title={Making gradient descent optimal for strongly convex stochastic optimization},
  author={Rakhlin, Alexander and Shamir, Ohad and Sridharan, Karthik},
  booktitle={International Conference on International Conference on Machine Learning},
  pages={1571--1578},
  year={2012}
}

@inproceedings{shamir2013stochastic,
  title={Stochastic gradient descent for non-smooth optimization: Convergence results and optimal averaging schemes},
  author={Shamir, Ohad and Zhang, Tong},
  booktitle={International Conference on Machine Learning},
  pages={71--79},
  year={2013},
  organization={PMLR}
}

@inproceedings{harvey2019tight,
  title={Tight analyses for non-smooth stochastic gradient descent},
  author={Harvey, Nicholas JA and Liaw, Christopher and Plan, Yaniv and Randhawa, Sikander},
  booktitle={Conference on Learning Theory},
  pages={1579--1613},
  year={2019},
  organization={PMLR}
}

@article{harvey2019simple,
  title={Simple and optimal high-probability bounds for strongly-convex stochastic gradient descent},
  author={Harvey, Nicholas JA and Liaw, Christopher and Randhawa, Sikander},
  journal={arXiv preprint arXiv:1909.00843},
  year={2019}
}

@article{jain2021making,
  title={Making the last iterate of {SGD} information theoretically optimal},
  author={Jain, Prateek and Nagaraj, Dheeraj M and Netrapalli, Praneeth},
  journal={SIAM Journal on Optimization},
  volume={31},
  number={2},
  pages={1108--1130},
  year={2021},
  publisher={SIAM}
}

@article{agarwal2012information,
  title={Information-theoretic lower bounds on the oracle complexity of stochastic convex optimization},
  author={Agarwal, Alekh and Bartlett, Peter L. and Ravikumar, Pradeep and Wainwright, Martin J.},
  journal={IEEE Transactions on Information Theory},
  volume={58},
  number={5},
  pages={3235--3249},
  year={2012}
}

@article{gorbunov2020stochastic,
  title={Stochastic optimization with heavy-tailed noise via accelerated gradient clipping},
  author={Gorbunov, Eduard and Danilova, Marina and Gasnikov, Alexander},
  journal={Advances in Neural Information Processing Systems},
  volume={33},
  pages={15042--15053},
  year={2020}
}

@article{cutkosky2021high,
  title={High-probability bounds for non-convex stochastic optimization with heavy tails},
  author={Cutkosky, Ashok and Mehta, Harsh},
  journal={Advances in Neural Information Processing Systems},
  volume={34},
  pages={4883--4895},
  year={2021}
}

@article{davis2021low,
  title={From low probability to high confidence in stochastic convex optimization},
  author={Davis, Damek and Drusvyatskiy, Dmitriy and Xiao, Lin and Zhang, Junyu},
  journal={Journal of Machine Learning Research},
  volume={22},
  number={49},
  pages={1--38},
  year={2021}
}

@book{ville1939,
  title     = {\'Etude Critique de la Notion de Collectif},
  author    = {Ville, Jean},
  year      = {1939},
  publisher = {Gauthier-Villars},
  address   = {Paris}
}

@article{darling1967confidence,
  title={Confidence sequences for mean, variance, and median},
  author={Darling, Donald A and Robbins, Herbert},
  journal={Proceedings of the National Academy of Sciences},
  volume={58},
  number={1},
  pages={66--68},
  year={1967}
}

@article{robbins1970boundary,
  title={Boundary crossing probabilities for the {W}iener process and sample sums},
  author={Robbins, Herbert and Siegmund, David},
  journal={The Annals of Mathematical Statistics},
  pages={1410--1429},
  year={1970},
  publisher={JSTOR}
}

@article{ramdas2025hypothesis,
  title   = {Hypothesis Testing with E-values},
  author  = {Ramdas, Aaditya and Wang, Ruodu},
  journal = {Foundations and Trends{\textregistered} in Statistics},
  volume  = {1},
  number  = {1--2},
  pages   = {1--390},
  year    = {2025}
}

@article{ramdas2023game,
  title={Game-theoretic statistics and safe anytime-valid inference},
  author={Ramdas, Aaditya and Gr{\"u}nwald, Peter and Vovk, Vladimir and Shafer, Glenn},
  journal={Statistical Science},
  volume={38},
  number={4},
  pages={576--601},
  year={2023},
  publisher={Institute of Mathematical Statistics}
}

@article{waudby2024anytime,
  title={Anytime-valid off-policy inference for contextual bandits},
  author={Waudby-Smith, Ian and Wu, Lili and Ramdas, Aaditya and Karampatziakis, Nikos and Mineiro, Paul},
  journal={ACM/IMS Journal of Data Science},
  volume={1},
  number={3},
  pages={1--42},
  year={2024},
  publisher={ACM New York, NY}
}

@article{feng2023anytime,
  title={The Anytime Convergence of Stochastic Gradient Descent with Momentum: From a Continuous-Time Perspective},
  author={Feng, Yasong and Jiang, Yifan and Wang, Tianyu and Ying, Zhiliang},
  journal={arXiv preprint arXiv:2310.19598},
  year={2023}
}

@article{yu2025root,
  title={The Root Finding Problem Revisited: Beyond the {R}obbins-{M}onro procedure},
  author={Yu, Yue and Banerjee, Moulinath and Ritov, Ya'acov},
  journal={arXiv preprint arXiv:2508.17591},
  year={2025}
}

@article{lan2012validation,
  title={Validation analysis of mirror descent stochastic approximation method},
  author={Lan, Guanghui and Nemirovski, Arkadi and Shapiro, Alexander},
  journal={Mathematical Programming},
  volume={134},
  number={2},
  pages={425--458},
  year={2012},
  publisher={Springer}
}

@incollection{benaim2006dynamics,
  title={Dynamics of stochastic approximation algorithms},
  author={Bena{\"\i}m, Michel},
  booktitle={Seminaire de Probabilites XXXIII},
  pages={1--68},
  year={2006},
  publisher={Springer}
}

@book{borkar2008stochastic,
  title={Stochastic Approximation: A Dynamical Systems Viewpoint},
  author={Borkar, Vivek S and Borkar, Vivek S},
  volume={100},
  year={2008},
  publisher={Springer}
}

@article{ghadimi2013stochastic,
  title={Stochastic first-and zeroth-order methods for nonconvex stochastic programming},
  author={Ghadimi, Saeed and Lan, Guanghui},
  journal={SIAM Journal on Optimization},
  volume={23},
  number={4},
  pages={2341--2368},
  year={2013},
  publisher={SIAM}
}

@article{ghadimi2016accelerated,
  title={Accelerated gradient methods for nonconvex nonlinear and stochastic programming},
  author={Ghadimi, Saeed and Lan, Guanghui},
  journal={Mathematical Programming},
  volume={156},
  number={1},
  pages={59--99},
  year={2016},
  publisher={Springer}
}

@article{bottou2018optimization,
  title={Optimization methods for large-scale machine learning},
  author={Bottou, L{\'e}on and Curtis, Frank E and Nocedal, Jorge},
  journal={SIAM review},
  volume={60},
  number={2},
  pages={223--311},
  year={2018},
  publisher={SIAM}
}

@inproceedings{reddi2016stochastic,
  title={Stochastic variance reduction for nonconvex optimization},
  author={Reddi, Sashank J and Hefny, Ahmed and Sra, Suvrit and Poczos, Barnabas and Smola, Alex},
  booktitle={International Conference on Machine Learning},
  pages={314--323},
  year={2016},
  organization={PMLR}
}

@inproceedings{lei2017less,
  title={Less than a single pass: Stochastically controlled stochastic gradient},
  author={Lei, Lihua and Jordan, Michael},
  booktitle={Artificial Intelligence and Statistics},
  pages={148--156},
  year={2017},
  organization={PMLR}
}

@inproceedings{nguyen2017sarah,
  title={SARAH: A novel method for machine learning problems using stochastic recursive gradient},
  author={Nguyen, Lam M and Liu, Jie and Scheinberg, Katya and Tak{\'a}{\v{c}}, Martin},
  booktitle={International Conference on Machine Learning},
  pages={2613--2621},
  year={2017},
  organization={PMLR}
}

@article{fang2018spider,
  title={Spider: Near-optimal non-convex optimization via stochastic path-integrated differential estimator},
  author={Fang, Cong and Li, Chris Junchi and Lin, Zhouchen and Zhang, Tong},
  journal={Advances in Neural Information Processing Systems},
  volume={31},
  year={2018}
}

@inproceedings{li2021page,
  title={PAGE: A simple and optimal probabilistic gradient estimator for nonconvex optimization},
  author={Li, Zhize and Bao, Hongyan and Zhang, Xiangliang and Richt{\'a}rik, Peter},
  booktitle={International Conference on Machine Learning},
  pages={6286--6295},
  year={2021},
  organization={PMLR}
}

@inproceedings{ge2015escaping,
  title={Escaping from saddle points—online stochastic gradient for tensor decomposition},
  author={Ge, Rong and Huang, Furong and Jin, Chi and Yuan, Yang},
  booktitle={Conference on Learning Theory},
  pages={797--842},
  year={2015},
  organization={PMLR}
}

@inproceedings{jin2017escape,
  title={How to escape saddle points efficiently},
  author={Jin, Chi and Ge, Rong and Netrapalli, Praneeth and Kakade, Sham M and Jordan, Michael I},
  booktitle={International Conference on Machine Learning},
  pages={1724--1732},
  year={2017},
  organization={PMLR}
}

@article{allen2018natasha,
  title={Natasha 2: Faster non-convex optimization than SGD},
  author={Allen-Zhu, Zeyuan},
  journal={Advances in Neural Information Processing Systems},
  volume={31},
  year={2018}
}

@article{zaheer2018adaptive,
  title={Adaptive methods for nonconvex optimization},
  author={Zaheer, Manzil and Reddi, Sashank and Sachan, Devendra and Kale, Satyen and Kumar, Sanjiv},
  journal={Advances in Neural Information Processing Systems},
  volume={31},
  year={2018}
}

@article{ward2020adagrad,
  title={Adagrad stepsizes: Sharp convergence over nonconvex landscapes},
  author={Ward, Rachel and Wu, Xiaoxia and Bottou, Leon},
  journal={Journal of Machine Learning Research},
  volume={21},
  number={219},
  pages={1--30},
  year={2020}
}

@article{kavis2022high,
  title={High probability bounds for a class of nonconvex algorithms with {A}dagrad stepsize},
  author={Kavis, Ali and Levy, Kfir Yehuda and Cevher, Volkan},
  journal={arXiv preprint arXiv:2204.02833},
  year={2022}
}

@inproceedings{li2022high,
  title={High probability guarantees for nonconvex stochastic gradient descent with heavy tails},
  author={Li, Shaojie and Liu, Yong},
  booktitle={International Conference on Machine Learning},
  pages={12931--12963},
  year={2022},
  organization={PMLR}
}

@article{fatkhullin2025can,
  title={Can SGD Handle Heavy-Tailed Noise?},
  author={Fatkhullin, Ilyas and H{\"u}bler, Florian and Lan, Guanghui},
  journal={arXiv preprint arXiv:2508.04860},
  year={2025}
}

@book{boucheron2013concentration,
    title = {Concentration Inequalities: A Nonasymptotic Theory of Independence},
    author = {Boucheron, Stéphane and Lugosi, Gábor and Massart, Pascal},
    publisher = {Oxford University Press},
    year = {2013}
}

@book{lan2020first,
  title={First-order and Stochastic Optimization Methods for Machine Learning},
  author={Lan, Guanghui},
  volume={1},
  year={2020},
  publisher={Springer}
}

@book{durrett2019probability,
  title={Probability: theory and examples},
  author={Durrett, Rick},
  volume={49},
  year={2019},
  publisher={Cambridge university press}
}

@book{hall2014martingale,
  title={Martingale Limit Theory and its Application},
  author={Hall, Peter and Heyde, Christopher C},
  year={2014},
  publisher={Academic press}
}

@inproceedings{garivier2016optimal,
  title={Optimal best arm identification with fixed confidence},
  author={Garivier, Aur{\'e}lien and Kaufmann, Emilie},
  booktitle={Conference on Learning Theory},
  pages={998--1027},
  year={2016},
  organization={PMLR}
}

@article{kaufmann2016complexity,
  title={On the complexity of best-arm identification in multi-armed bandit models},
  author={Kaufmann, Emilie and Capp{\'e}, Olivier and Garivier, Aur{\'e}lien},
  journal={The Journal of Machine Learning Research},
  volume={17},
  number={1},
  pages={1--42},
  year={2016},
  publisher={JMLR. org}
}

@article{kaufmann2021mixture,
  title={Mixture martingales revisited with applications to sequential tests and confidence intervals},
  author={Kaufmann, Emilie and Koolen, Wouter M},
  journal={Journal of Machine Learning Research},
  volume={22},
  number={246},
  pages={1--44},
  year={2021}
}

@inproceedings{agrawal2020optimal,
  title={Optimal $\delta$-Correct Best-Arm Selection for Heavy-Tailed Distributions},
  author={Agrawal, Shubhada and Juneja, Sandeep and Glynn, Peter},
  booktitle={Algorithmic Learning Theory},
  pages={61--110},
  year={2020},
  organization={PMLR}
}

@article{agrawal2021optimal,
  title={Optimal best-arm identification methods for tail-risk measures},
  author={Agrawal, Shubhada and Koolen, Wouter M and Juneja, Sandeep},
  journal={Advances in neural information processing systems},
  volume={34},
  pages={25578--25590},
  year={2021}
}

%--------------------------------------------------------------------------------
%--------------------------------------------------------------------------------
%--------------------------------------------------------------------------------

\appendix

\section{Additional proofs}
\label{sec:additional-proofs}

%--------------------------------------------------------------------------------

\subsection{Proof of Proposition~\ref{prop:Ut-SA}}
Write the observable bound from Corollary~\ref{cor:observable-CS} as
\[
    U_t^{\mathrm{obs}}(\alpha)
    \;=\;
    \frac{1}{2S_t}\Big(
        A_t(\alpha) \;+\; R_x^2 \;+\; B_t
    \Big),
\]
where
\[
    A_t(\alpha)
    :=
    7\sqrt{\widetilde{\Sigma}_{t,\mathrm{eff}}^2\Big(\log\tfrac{2}{\alpha}
    +\log\log(\mathrm e+\widetilde{\Sigma}_{t,\mathrm{eff}}^2)\Big)},
    \qquad
    B_t := \sum_{s=1}^t \eta_s^2\|g_s\|^2 .
\]

\medskip
\noindent\emph{Almost sure rate.}
Under the stochastic approximation stepsize, we have $V_t=\sum_{s=1}^t \eta_s^2$ bounded and therefore
\[
    \sup_{t\ge 1}\widetilde{\Sigma}_t^2
    = \sigma^2 R_x^2 \sup_{t\ge 1} V_t
    < \infty,
\]
which implies that $\widetilde{\Sigma}_{t,\mathrm{eff}}^2$ (and thus $A_t(\alpha)$) is uniformly bounded in $t$. Moreover, by Lemma~\ref{lem:sum-eta2-g2-finite}, $B_t$ converges almost surely to a finite limit. Therefore the numerator $A_t(\alpha)+R_x^2+B_t$ is almost surely bounded, and since $S_t\to\infty$, we obtain $U_t^{\mathrm{obs}}(\alpha)=O(1/S_t)$ almost surely, and thus $U_t^{\mathrm{obs}}(\alpha)\to 0$.

\medskip
\noindent\emph{Expected rate.}
The same uniform boundedness gives $\sup_t A_t(\alpha)<\infty$. Furthermore, by the tower property and Assumption~\ref{as:gradients:moment},
\[
    \mathbb{E}[B_t]
    =
    \sum_{s=1}^t \eta_s^2 \mathbb{E}[\|g_s\|^2]
    \le
    \sum_{s=1}^t \eta_s^2 \,\mathbb{E}\!\big[\mathbb{E}[\|g_s\|^2\mid\mathcal F_{s-1}]\big]
    \le
    M^2 \sum_{s=1}^\infty \eta_s^2
    <\infty.
\]
Hence there exists a constant $C<\infty$ such that $\mathbb{E}[A_t(\alpha)+R_x^2+B_t]\le C$ for all $t$, and
\[
    \mathbb{E}[U_t^{\mathrm{obs}}(\alpha)]
    \le
    \frac{C}{2S_t}
    =
    O(1/S_t),
\]
which concludes the proof.

%--------------------------------------------------------------------------------

\subsection{Proof of Proposition~\ref{prop:Ut-poly}}
Write the observable bound from Corollary~\ref{cor:observable-CS} as
\[
    U_t^{\mathrm{obs}}(\alpha)
    \;=\;
    \frac{1}{2S_t}\Big(
        A_t(\alpha) \;+\; R_x^2 \;+\; B_t
    \Big),
\]
where
\[
    A_t(\alpha)
    :=
    7\sqrt{\widetilde{\Sigma}_{t,\mathrm{eff}}^2\Big(\log\tfrac{2}{\alpha}
    +\log\log(\mathrm e+\widetilde{\Sigma}_{t,\mathrm{eff}}^2)\Big)},
    \qquad
    B_t := \sum_{s=1}^t \eta_s^2\|g_s\|^2 .
\]

\medskip
\noindent\emph{Expected rate.}
Since $x\mapsto x^{-1/2}$ is nonincreasing, for $\eta_s=\eta_0 s^{-1/2}$ we have
$s^{-1/2}\ge t^{-1/2}$ for all $1\le s\le t$, hence
\begin{equation}\label{eq:St-lower}
    S_t=\sum_{s=1}^t \eta_s
    =\eta_0\sum_{s=1}^t s^{-1/2}
    \ge
    \eta_0\sum_{s=1}^t t^{-1/2}
    =
    \eta_0\sqrt{t}.
\end{equation}
Moreover, by the integral comparison for $x\mapsto 1/x$,
\begin{equation}\label{eq:Vt-upper}
    V_t=\sum_{s=1}^t \eta_s^2
    =\eta_0^2\sum_{s=1}^t \frac{1}{s}
    \le
    \eta_0^2\Big(1+\log t\Big).
\end{equation}
Recalling $\widetilde{\Sigma}_{t,\mathrm{eff}}^2 =\max\!\{\sigma^2R_x^2V_t,\;2(\log\tfrac{2}{\alpha}+1)\}$ and using $\max\{a,b\}\le a+b$ for $a,b\ge0$, we obtain
\begin{equation}\label{eq:Sigmaeff-upper}
    \widetilde{\Sigma}_{t,\mathrm{eff}}^2
    \le
    \sigma^2R_x^2V_t + 2\Big(\log\tfrac{2}{\alpha}+1\Big)
    \le
    C_1\bigl(1+\log t\bigr),
\end{equation}
where $C_1<\infty$ depends only on $(\alpha,\sigma,R_x,\eta_0)$. We now bound $A_t(\alpha)$. For any $u\ge0$, $\log\log(\mathrm e+u)\le \log(\mathrm e+u)$
since $\log(\mathrm e+u)\ge 1$. Therefore, for all $t\ge1$,
\[
    A_t(\alpha)
    \le
    7\sqrt{\widetilde{\Sigma}_{t,\mathrm{eff}}^2\Big(\log\tfrac{2}{\alpha}
    +\log(\mathrm e+\widetilde{\Sigma}_{t,\mathrm{eff}}^2)\Big)}.
\]
For $t\ge e$, \eqref{eq:Sigmaeff-upper} implies
\[
    \mathrm e+\widetilde{\Sigma}_{t,\mathrm{eff}}^2
    \le
    \mathrm e + C_1(1+\log t)
    \le
    C_2\log t,
\]
with $C_2:=\mathrm e+2C_1$, and hence $\log(\mathrm e+\widetilde{\Sigma}_{t,\mathrm{eff}}^2)\le \log C_2+\log\log t\le \log C_2+\log t$.
Combining this with \eqref{eq:Sigmaeff-upper} yields that for all $t\ge e$,
\[
    A_t(\alpha)
    \le
    7\sqrt{C_1(1+\log t)\Big(\log\tfrac{2}{\alpha}+\log C_2+\log t\Big)}
    \le
    C_3(1+\log t),
\]
for some constant $C_3<\infty$ depending only on $(\alpha,\sigma,R_x,\eta_0)$.

Next, by the tower property and Assumption~\ref{as:gradients:moment},
\[
    \mathbb{E}[B_t]
    =
    \sum_{s=1}^t \eta_s^2\,\mathbb{E}[\|g_s\|^2]
    \le
    \sum_{s=1}^t \eta_s^2\,\mathbb{E}\!\big[\mathbb{E}[\|g_s\|^2\mid\mathcal F_{s-1}]\big]
    \le
    M^2\sum_{s=1}^t \eta_s^2
    =
    M^2V_t
    \le
    M^2\eta_0^2(1+\log t),
\]
where the last inequality uses \eqref{eq:Vt-upper}.

Taking expectations in the decomposition of $U_t^{\mathrm{obs}}(\alpha)$ and using the bounds above, for all $t\ge e$,
\[
    \mathbb{E}\!\left[U_t^{\mathrm{obs}}(\alpha)\right]
    =
    \frac{1}{2S_t}\Big(A_t(\alpha)+R_x^2+\mathbb{E}[B_t]\Big)
    \le
    \frac{1}{2\eta_0\sqrt{t}}
    \Big(C_3(1+\log t)+R_x^2+M^2\eta_0^2(1+\log t)\Big),
\]
where we used \eqref{eq:St-lower}. Absorbing constants, this yields
\[
    \mathbb{E}\!\left[U_t^{\mathrm{obs}}(\alpha)\right]
    \le
    \frac{C_4(1+\log t)}{\sqrt{t}}
    =
    O\left(\frac{\log t}{\sqrt{t}}\right),
\]
as claimed.

\medskip
\noindent\emph{Almost sure rate.}
We reuse \eqref{eq:St-lower}, \eqref{eq:Vt-upper}, and the bound $A_t(\alpha)\le C_3(1+\log t)$ proved in part \textup{(i)}. Thus it suffices to show $B_t=O(\log t)$ almost surely. Using $\|a+b\|^2\le 2\|a\|^2+2\|b\|^2$ and $\|\nabla f(x_s)\|\le G$ almost surely,
\[
    B_t=\sum_{s=1}^t \eta_s^2\|g_s\|^2
    \le 2G^2V_t + 2\sum_{s=1}^t \eta_s^2\|\xi_s\|^2.
\]
Since $V_t=O(\log t)$ deterministically, it remains to control $\sum_{s=1}^t\eta_s^2\|\xi_s\|^2$. Define the martingale differences
\[
    \Delta_s:=\eta_s^2\Big(\|\xi_s\|^2-\mathbb E[\|\xi_s\|^2\mid\mathcal F_{s-1}]\Big),
    \qquad
    M_t:=\sum_{s=1}^t \Delta_s.
\]
Then $\{M_t\}$ is an $\{\mathcal F_t\}$-martingale. Moreover,
\[
    \mathbb E[\Delta_s^2\mid\mathcal F_{s-1}]
    =\eta_s^4\,\mathbb E\!\left[\Big(\|\xi_s\|^2-\mathbb E[\|\xi_s\|^2\mid\mathcal F_{s-1}]\Big)^2\Bigm|\mathcal F_{s-1}\right]
    =\eta_s^4\,\mathrm{Var}\!\big(\|\xi_s\|^2\mid\mathcal F_{s-1}\big).
\]
Since $\mathrm{Var}(Y\mid\mathcal F)\le \mathbb E[Y^2\mid\mathcal F]$ for any $Y$, it follows that
\[
    \mathbb E[\Delta_s^2\mid\mathcal F_{s-1}]
    \le \eta_s^4\,\mathbb E[\|\xi_s\|^4\mid\mathcal F_{s-1}]
    \le 16d^2\sigma^4\,\eta_s^4
    \qquad\text{almost surely},
\]
where the last inequality follows from Lemma~\ref{lem:subg-moments}. Therefore,
\[
    \sup_{t\ge 1}\mathbb E[M_t^2]
    =
    \sup_{t\ge 1}\sum_{s=1}^t \mathbb E[\Delta_s^2]
    \le 16d^2\sigma^4\sum_{s=1}^\infty \eta_s^4
    <\infty,
\]
since $\eta_s=\eta_0\, s^{-1/2}$ implies $\sum_s\eta_s^4<\infty$. Hence $\{M_t\}$ is $L^2$-bounded. By the martingale convergence theorem, there exists $M_\infty\in L^2$ such that $M_t\to M_\infty$ almost surely; in particular, $\sup_{t\ge1}|M_t|<\infty$ almost surely. Finally,
\[
    \sum_{s=1}^t \eta_s^2\|\xi_s\|^2
    =
    \sum_{s=1}^t \eta_s^2\,\mathbb E[\|\xi_s\|^2\mid\mathcal F_{s-1}]
    + M_t
    \le
    4d\sigma^2 V_t + O(1)
    =
    O(\log t)
    \qquad\text{almost surely},
\]
again by Lemma~\ref{lem:subg-moments} and $V_t=O(\log t)$. Thus $B_t=O(\log t)$ almost surely, and plugging into the decomposition of
$U_t^{\mathrm{obs}}(\alpha)$ together with $S_t\ge \eta_0\sqrt t$ and $A_t(\alpha)=O(\log t)$ (from part \textup{(i)}) yields
\[
    U_t^{\mathrm{obs}}(\alpha)=O\left(\frac{\log t}{\sqrt t}\right)
    \qquad\text{almost surely.}
\]

%--------------------------------------------------------------------------------

\subsection{Proof of Theorem~\ref{thm:stopping-certified-main}}
Let
\[
    \mathcal E_\alpha
    :=
    \Bigl\{
        \forall t\ge 1:\ f(\bar x_t)-f(x^\star) \le U_t^{\mathrm{obs}}(\alpha)
    \Bigr\}.
\]
By \eqref{eq:cs-stopping-recall}, $\mathbb P(\mathcal E_\alpha^{\mathrm c})\le \alpha$. Fix $m\ge 1$. On $\mathcal E_\alpha\cap\{\tau_\varepsilon=m\}$,
\[
    f(\bar x_m)-f(x^\star) \le U_m^{\mathrm{obs}}(\alpha)\le \varepsilon,
\]
where the first inequality is the definition of $\mathcal E_\alpha$ at $t=m$ and the second follows from $\tau_\varepsilon=m$. Hence
\[
    \{\tau_\varepsilon=m\}\cap\{f(\bar x_m)-f(x^\star)>\varepsilon\}
    \subseteq \mathcal E_\alpha^{\mathrm c}.
\]
Taking the union over $m$ and using $\{\tau_\varepsilon<\infty\}=\bigcup_{m\ge 1}\{\tau_\varepsilon=m\}$ yields
\[
    \{\tau_\varepsilon<\infty\}\cap\{f(\bar x_{\tau_\varepsilon})-f(x^\star)>\varepsilon\}
    \subseteq \mathcal E_\alpha^{\mathrm c}.
\]
Therefore,
\[
\mathbb{P}\!\left(
    \{\tau_\varepsilon<\infty\}\cap\{f(\bar x_{\tau_\varepsilon})-f(x^\star) > \varepsilon\}
\right)
\le \mathbb P(\mathcal E_\alpha^{\mathrm c})
\le \alpha.
\]

%--------------------------------------------------------------------------------

\subsection{Proof of Proposition~\ref{prop:stopping-certified}}
We first establish almost sure finiteness of the stopping time $\tau_\varepsilon$.

\medskip
\noindent\emph{Almost sure termination.}
Under the stochastic approximation conditions~\eqref{eq:stepsizes:SA} and Assumption~\ref{as:gradients:moment}, Proposition~\ref{prop:Ut-SA} shows that
\[
    U_t^{\mathrm{obs}}(\alpha)\;\to\;0
    \qquad\text{almost surely as } t\to\infty.
\]
Define the event
\[
    \Omega_0
    :=
    \Big\{
        \lim_{t\to\infty} U_t^{\mathrm{obs}}(\alpha) = 0
    \Big\},
\]
so that $\mathbb{P}(\Omega_0)=1$. Fix $\varepsilon>0$ and $\omega\in\Omega_0$. By definition of almost sure convergence, there exists a (random, finite) index $T_\varepsilon(\omega)$ such that
\[
    U_t^{\mathrm{obs}}(\alpha)(\omega) \le \varepsilon
    \qquad \forall\, t\ge T_\varepsilon(\omega).
\]
By the definition~\eqref{eq:tau-def-stopping} of $\tau_\varepsilon$,
\[
    \tau_\varepsilon(\omega)
    =
    \inf\bigl\{ t\ge 1 :
    U_t^{\mathrm{obs}}(\alpha)(\omega)\le\varepsilon \bigr\}
    \;\le\;
    T_\varepsilon(\omega)
    \;<\;
    \infty.
\]
Since this holds for every $\omega\in\Omega_0$ and $\mathbb{P}(\Omega_0)=1$, we conclude that $\tau_\varepsilon<\infty$ almost surely.

Under polynomial stepsizes $\eta_t=\eta_0\,t^{-1/2}$ with uniformly bounded gradients, Proposition~\ref{prop:Ut-poly} likewise yields $U_t^{\mathrm{obs}}(\alpha)\to0$ almost surely, and the same argument applies.

\medskip
\noindent\emph{Validity at the stopping time.}
Define the event
\[
    \mathcal{E}_\alpha
    :=
    \Big\{
        \forall t\ge 1:\ 
        f(\bar x_t)-f(x^\star)\le U_t^{\mathrm{obs}}(\alpha)
    \Big\}.
\]
By~\eqref{eq:cs-stopping-recall}, $\mathbb{P}(\mathcal{E}_\alpha)\ge 1-\alpha$. On the event $\mathcal{E}_\alpha$, the inequality $f(\bar x_t)-f(x^\star)\le U_t^{\mathrm{obs}}(\alpha)$ holds for all $t\ge1$, and in particular at the (random) time $\tau_\varepsilon$. Hence, on $\mathcal{E}_\alpha$,
\[
    f(\bar x_{\tau_\varepsilon})-f(x^\star)
    \;\le\;
    U_{\tau_\varepsilon}^{\mathrm{obs}}(\alpha).
\]
By definition of $\tau_\varepsilon$, we have $U_{\tau_\varepsilon}^{\mathrm{obs}}(\alpha)\le\varepsilon$, and therefore
\[
    f(\bar x_{\tau_\varepsilon})-f(x^\star)\le\varepsilon
    \qquad \text{on }\mathcal{E}_\alpha.
\]
Thus $\mathcal{E}_\alpha\subseteq\{f(\bar x_{\tau_\varepsilon}) -f(x^\star)\le\varepsilon\}$ and taking probabilities yields
\[
    \mathbb{P}\Big(
        f(\bar x_{\tau_\varepsilon})-f(x^\star)\le\varepsilon
    \Big)
    \;\ge\;
    \mathbb{P}(\mathcal{E}_\alpha)
    \;\ge\;
    1-\alpha.
\]
This establishes~\eqref{eq:stopped-eps-opt} and completes the proof.

%--------------------------------------------------------------------------------

\subsection{Proof of Lemma~\ref{lem:Etau-from-envelope}}
Fix an outcome $\omega$ in the almost sure event on which \textup{(ii)} holds for all $t\ge1$. If $t\ge T(\varepsilon/K(\omega))$, then by definition of $T$ and the monotonicity of $b$ we have $b(t)\le \varepsilon/K(\omega)$, and thus $U_t^{\mathrm{obs}}(\alpha)(\omega)\le K(\omega)b(t)\le\varepsilon$. By definition of $\tau_\varepsilon$ as the first time the certificate falls below $\varepsilon$, this implies $\tau_\varepsilon(\omega)\le T(\varepsilon/K(\omega))$. Taking expectations yields the claim.

%--------------------------------------------------------------------------------

\subsection{Proof of Theorem~\ref{thm:Etau-sqrt}}
We bound the stopping time by constructing an almost sure envelope of the form $U_t^{\mathrm{obs}}(\alpha)\le K_1\,b(t)$ with $b(t)\asymp (1+\log t)/\sqrt t$ and a random prefactor $K_1$ having sufficiently high moments, then invoke Lemma~\ref{lem:Etau-from-envelope}. Recall the decomposition (as in the proof of Proposition~\ref{prop:Ut-poly})
\[
    U_t^{\mathrm{obs}}(\alpha)
    =
    \frac{1}{2S_t}\Big(A_t(\alpha)+R_x^2+B_t\Big),
    \qquad
    B_t=\sum_{s=1}^t \eta_s^2\|g_s\|^2,
\]
and the definitions of $A_t(\alpha)$ and $\widetilde{\Sigma}_{t,\mathrm{eff}}^2$ used there. For the square-root schedule $\eta_s=\eta_0\,s^{-1/2}$, the deterministic estimates \eqref{eq:St-lower} and \eqref{eq:Vt-upper} from Proposition~\ref{prop:Ut-poly} yield
\begin{equation}
\label{eq:St-Vt-sqrt-recall}
    S_t \ge \eta_0\sqrt t,
    \qquad
    V_t \le \eta_0^2(1+\log t).
\end{equation}
Moreover, the bound on the time-uniform term $A_t(\alpha)$ established in the same proof implies that there exists $C_A<\infty$ such that
\begin{equation}
\label{eq:A-sqrt-recall}
    A_t(\alpha)\le C_A(1+\log t)
    \qquad\text{for all } t\ge 1.
\end{equation}
We next control $B_t$ pathwise. Writing $g_s=\nabla f(x_s)+\xi_s$ and using $\|a+b\|^2\le 2\|a\|^2+2\|b\|^2$ together with Assumption~\ref{as:bounded-grad}, we obtain
\[
    B_t
    =
    \sum_{s=1}^t \eta_s^2\|g_s\|^2
    \le
    2\sum_{s=1}^t \eta_s^2\|\nabla f(x_s)\|^2
    +
    2\sum_{s=1}^t \eta_s^2\|\xi_s\|^2
    \le
    2G^2V_t + 2\sum_{s=1}^t \eta_s^2\|\xi_s\|^2.
\]
Decompose the weighted noise-square sum into its predictable compensator plus a martingale:
\[
    \sum_{s=1}^t \eta_s^2\|\xi_s\|^2
    =
    \sum_{s=1}^t \eta_s^2\,\mathbb{E}[\|\xi_s\|^2\mid\mathcal{F}_{s-1}]
    + M_t,
    \qquad
    M_t:=\sum_{s=1}^t \Delta_s,
\]
where $\Delta_s:=\eta_s^2(\|\xi_s\|^2-\mathbb{E}[\|\xi_s\|^2\mid \mathcal{F}_{s-1}])$. By Lemma~\ref{lem:subg-moments}, $\mathbb{E}[\|\xi_s\|^2\mid\mathcal{F}_{s-1}]\le 4d\sigma^2$ almost surely, hence
\[
    \sum_{s=1}^t \eta_s^2\,\mathbb{E}[\|\xi_s\|^2\mid\mathcal{F}_{s-1}]
    \le
    4d\sigma^2 V_t
    \le
    4d\sigma^2\eta_0^2(1+\log t).
\]
Combining the last two displays and using $V_t\le \eta_0^2(1+\log t)$ again gives that there exists $C_B<\infty$ (depending only on $(G,\sigma,d,\eta_0)$) such that
\begin{equation}\label{eq:Bt-envelope}
    B_t \le C_B(1+\log t) + 2\sup_{s\ge1}|M_s|.
\end{equation}

We now introduce the random variable $K:=1+\sup_{s\ge1}|M_s|$. By Lemma~\ref{lem:Mt-maximal}, $\mathbb{E}[K^4]<\infty$, and in particular $K<\infty$ almost surely. Using
\eqref{eq:A-sqrt-recall}, \eqref{eq:Bt-envelope}, and $S_t\ge \eta_0\sqrt t$, we obtain that there exists $C_0<\infty$ depending only on $(\alpha,\eta_0,R_x,G,\sigma,d)$ such that for all $t\ge 1$,
\begin{equation}
\label{eq:U-envelope-final}
    U_t^{\mathrm{obs}}(\alpha)
    \le
    C_0\,K\,\frac{1+\log t}{\sqrt t}
    \qquad\text{almost surely.}
\end{equation}

Define $b(t):=(1+\log t)/\sqrt t$ for $t\ge 1$. Then $b(t)\to 0$ as $t\to\infty$, and $b$ is nonincreasing for all sufficiently large $t$. Recall the definition $T(u):=\inf\{t\ge1:\ b(t)\le u\}$. As shown in Lemma~\ref{lem:invert-b}, there exists a universal constant $C_1<\infty$ such that for all $u\in(0,1/2)$,
\begin{equation}\label{eq:T-inv-again}
    T(u)\le C_1\,u^{-2}\,\log^2\!\frac{1}{u}.
\end{equation}
Set $K_1:=C_0 K \ge 1$. From \eqref{eq:U-envelope-final} we have $U_t^{\mathrm{obs}}(\alpha)\le K_1\,b(t)$ almost surely, and therefore
Lemma~\ref{lem:Etau-from-envelope} yields
\[
    \mathbb{E}[\tau_\varepsilon]
    \le
    \mathbb{E}\!\left[T\!\left(\frac{\varepsilon}{K_1}\right)\right].
\]
Using \eqref{eq:T-inv-again} and $\varepsilon\in(0,1/2)$ gives
\[
    \mathbb{E}[\tau_\varepsilon]
    \le
    C_1\,\varepsilon^{-2}\,
    \mathbb{E}\!\left[
        K_1^{2}\,
        \log^2\!\left(\frac{K_1}{\varepsilon}\right)
    \right].
\]
To bound the expectation uniformly in $\varepsilon$, note that $\log(K_1/\varepsilon) \le \log(1/\varepsilon)+\log K_1$, so
\[
    K_1^{2}\log^2\!\left(\frac{K_1}{\varepsilon}\right)
    \le
    2K_1^2 \log^2\!\frac{1}{\varepsilon}
    +
    2K_1^2 \log^2(K_1).
\]
The first term has expectation $2\log^2(1/\varepsilon)\,\mathbb E[K_1^2]<\infty$. 

For the second term, fix any $\delta\in(0,1)$. Since $x^\delta$ dominates $\log^2 x$ as $x\to\infty$ and $x\mapsto \log^2 x / x^\delta$ is continuous on $[1,\infty)$, there exists $C_\delta<\infty$ such that $\log^2 x\le C_\delta x^\delta$ for all $x\ge 1$. Because $K_1\ge 1$ almost surely, this yields $K_1^2\log^2(K_1)\le C_\delta K_1^{2+\delta}$ almost surely. Because $\mathbb E[K_1^{2+\delta}]<\infty$ (which holds since $\mathbb E[K^4]<\infty$ and $\delta\le 2$), it follows that
$\mathbb E[K_1^2\log^2(K_1)]<\infty$. Therefore there exists a constant $C_2<\infty$, independent of $\varepsilon$, such that
\[
    \mathbb{E}\!\left[
        K_1^{2}\,
        \log^2\!\left(\frac{K_1}{\varepsilon}\right)
    \right]
    \le
    C_2\,\log^2\!\frac{1}{\varepsilon}.
\]
Combining the last two displays yields
\[
    \mathbb{E}[\tau_\varepsilon]
    \le
    C\,\varepsilon^{-2}\,\log^2\!\Big(\frac{1}{\varepsilon}\Big),
\]
after absorbing constants into $C$, which concludes the proof.

%--------------------------------------------------------------------------------
%--------------------------------------------------------------------------------
%--------------------------------------------------------------------------------

\section{Supporting lemmas}
\label{sec:supporting:proofs}

\begin{lemma}
\label{lem:sum-eta2-g2-finite}
Let Assumption~\ref{as:gradients:moment} hold and assume that $\sum_{t=1}^\infty \eta_t^2 < \infty$. Then,
\[
    \sum_{t=1}^\infty \eta_t^2 \|g_t\|^2 < \infty
    \qquad \text{almost surely}.
\]
\end{lemma}
\begin{proof}
Let $Y_n:=\sum_{t=1}^n \eta_t^2\|g_t\|^2$ and $Y:=\lim_{n\to\infty}Y_n\in[0,\infty]$.
By monotone convergence,
\[
    \mathbb{E}[Y]=\lim_{n\to\infty}\mathbb{E}[Y_n]
    =\lim_{n\to\infty}\sum_{t=1}^n \eta_t^2\,\mathbb{E}[\|g_t\|^2].
\]
By the tower property and Assumption~\ref{as:gradients:moment},
$\mathbb{E}[\|g_t\|^2]=\mathbb{E}[\mathbb{E}[\|g_t\|^2\mid\mathcal F_{t-1}]]\le M^2$.
Hence
\[
    \mathbb{E}[Y]\le M^2\sum_{t=1}^\infty \eta_t^2<\infty.
\]
Since $Y\ge0$, $\mathbb{E}[Y]<\infty$ implies $\mathbb{P}(Y<\infty)=1$.
\end{proof}

\begin{lemma}
\label{lem:subg-moments}
Suppose Assumption~\ref{as:gradients} holds. Then, for every $t\ge 1$,\footnote{The constants in this lemma are not optimized. Under \eqref{eq:cond-subg-noise}, a sharper moment computation yields $\mathbb E[\|\xi_t\|^2\mid\mathcal F_{t-1}]\le d\sigma^2$, $\mathbb E[\|\xi_t\|^4\mid\mathcal F_{t-1}]\le 3d^2\sigma^4$, and $\mathbb E[\|\xi_t\|^8\mid\mathcal F_{t-1}]\le 105\,d^4\sigma^8$. We keep the present bounds for simplicity.}
\[
    \mathbb E[\|\xi_t\|^2\mid\mathcal F_{t-1}] \le 4 d\sigma^2,
    \quad
    \mathbb E[\|\xi_t\|^4\mid\mathcal F_{t-1}] \le 16d^2\sigma^4,
    \quad
    \mathbb E[\|\xi_t\|^8\mid\mathcal F_{t-1}] \le 768\,d^4\sigma^8
    \qquad\text{almost surely.}
\]
\end{lemma}
\begin{proof}
Fix $t\ge1$ and $i\in\{1,\dots,d\}$. Taking $u=e_i$ in \eqref{eq:cond-subg-noise} yields
\[
    \mathbb E\!\left[e^{\lambda\xi_{t,i}}\mid\mathcal F_{t-1}\right]
    \le \exp\!\left(\frac{\lambda^2\sigma^2}{2}\right)
    \qquad \forall \lambda\in\mathbb R.
\]
By a standard Chernoff bound (see \cite[Section~2.3]{boucheron2013concentration}), this implies
\[
    \mathbb P(\xi_{t,i}>x\mid\mathcal F_{t-1})
    \ \vee\
    \mathbb P(-\xi_{t,i}>x\mid\mathcal F_{t-1})
    \le
    \exp\!\left(-\frac{x^2}{2\sigma^2}\right),
    \qquad \forall x>0.
\]
Thus, conditional on $\mathcal F_{t-1}$, $\xi_{t,i}$ satisfies the assumption in \cite[Theorem~2.1]{boucheron2013concentration} with $v=\sigma^2$. Applying that theorem with $q=1$, $q=2$, and $q=4$ yields
\[
    \mathbb E[\xi_{t,i}^2\mid\mathcal F_{t-1}] \le 4\sigma^2,
    \qquad
    \mathbb E[\xi_{t,i}^4\mid\mathcal F_{t-1}] \le 16\sigma^4,
    \qquad
    \mathbb E[\xi_{t,i}^8\mid\mathcal F_{t-1}] \le 768\,\sigma^8,
\]
almost surely. Summing the $q=1$ inequality over $i$ gives
\[
    \mathbb E[\|\xi_t\|^2\mid\mathcal F_{t-1}]
    =
    \sum_{i=1}^d \mathbb E[\xi_{t,i}^2\mid\mathcal F_{t-1}]
    \le
    4d\sigma^2.
\]
Next, using $(\sum_{i=1}^d a_i)^2\le d\sum_{i=1}^d a_i^2$ with $a_i=\xi_{t,i}^2 \ge 0$, 
\[
    \|\xi_t\|^4
    =
    \Big(\sum_{i=1}^d \xi_{t,i}^2\Big)^2
    \le
    d\sum_{i=1}^d \xi_{t,i}^4,
\]
and hence
\[
    \mathbb E[\|\xi_t\|^4\mid\mathcal F_{t-1}]
    \le
    d\sum_{i=1}^d \mathbb E[\xi_{t,i}^4\mid\mathcal F_{t-1}]
    \le
    16d^2\sigma^4.
\]
Finally, using $(\sum_{i=1}^d a_i)^4\le d^3\sum_{i=1}^d a_i^4$ with $a_i=\xi_{t,i}^2 \ge 0$,
\[
    \|\xi_t\|^8
    =
    \Big(\sum_{i=1}^d \xi_{t,i}^2\Big)^4
    \le
    d^3\sum_{i=1}^d \xi_{t,i}^8,
\]
and therefore
\[
    \mathbb E[\|\xi_t\|^8\mid\mathcal F_{t-1}]
    \le
    d^3\sum_{i=1}^d \mathbb E[\xi_{t,i}^8\mid\mathcal F_{t-1}]
    \le
    d^3 d\, 768\,\sigma^8
    =
    768\,d^4\sigma^8.
\]
This concludes the proof.
\end{proof}

\begin{lemma}
\label{lem:Mt-maximal}
Suppose Assumption~\ref{as:gradients} holds. Let $\{\eta_t\}_{t\ge1}$ be a deterministic stepsize sequence satisfying $\sum_{t=1}^\infty \eta_t^4 < \infty$, and define the adapted process
$\{M_t\}_{t\ge1}$ by
\[
    M_t := \sum_{s=1}^t \Delta_s,
    \qquad
    \Delta_s := \eta_s^2\Big(\|\xi_s\|^2
    - \mathbb{E}[\|\xi_s\|^2\mid\mathcal{F}_{s-1}]\Big).
\]
Then $\{M_t\}$ is an $\{\mathcal F_t\}$-martingale and
\[
    \mathbb{E}\!\left[\sup_{t\ge1}|M_t|^4\right] < \infty.
\]
\end{lemma}
\begin{proof}
Since $\Delta_s$ is $\mathcal{F}_s$-measurable, integrable, and satisfies $\mathbb{E}[\Delta_s\mid\mathcal{F}_{s-1}]=0$ almost surely by construction, it follows that $\{M_t\}_{t\ge 1}$ is an $\{\mathcal{F}_t\}$-martingale with differences $\{\Delta_s\}_{s\ge1}$. Fix $T\ge1$. By Doob's $L^4$ maximal inequality for martingales (e.g., \cite[Theorem~4.4.4]{durrett2019probability}),
\begin{equation}
\label{eq:doob4}
    \mathbb{E}\!\left[\sup_{1\le t\le T}|M_t|^4\right]
    \le
    \Big(\frac{4}{3}\Big)^4\,\mathbb{E}\!\left[|M_T|^4\right].
\end{equation}
Next, Burkholder's square-function inequality (upper bound in \cite[Theorem~2.10]{hall2014martingale}) applied to the martingale $\{M_t\}_{t=0}^T$ with differences $\Delta_s$ yields an (explicit) constant $C_1<\infty$ such that
\begin{equation}
\label{eq:burkholder4}
    \mathbb{E}\!\left[|M_T|^4\right]
    \le
    C_1\,\mathbb{E}\!\left[\Big(\sum_{s=1}^T \Delta_s^2\Big)^2\right].
\end{equation}
Combining \eqref{eq:doob4} and \eqref{eq:burkholder4} gives
\begin{equation}
\label{eq:max4-via-square}
    \mathbb{E}\!\left[\sup_{1\le t\le T}|M_t|^4\right]
    \le
    \Big(\frac{4}{3}\Big)^4 C_1\,
    \mathbb{E}\!\left[\Big(\sum_{s=1}^T \Delta_s^2\Big)^2\right].
\end{equation}
Since $\Delta_s^2\ge0$, Minkowski's inequality in $L^2$ implies
\[
    \Big\|\sum_{s=1}^T \Delta_s^2\Big\|_{L^2}
    \le
    \sum_{s=1}^T \|\Delta_s^2\|_{L^2}
    =
    \sum_{s=1}^T \|\Delta_s\|_{L^4}^2,
\]
where $\|\Delta_s^2\|_{L^2}=(\mathbb{E}|\Delta_s|^4)^{1/2}=\|\Delta_s\|_{L^4}^2$. Therefore,
\begin{equation}
\label{eq:minkowski-L2}
    \mathbb{E}\!\left[\Big(\sum_{s=1}^T \Delta_s^2\Big)^2\right]
    =
    \Big\|\sum_{s=1}^T \Delta_s^2\Big\|_{L^2}^2
    \le
    \Big(\sum_{s=1}^T (\mathbb{E}|\Delta_s|^4)^{1/2}\Big)^2.
\end{equation}
To bound $\mathbb{E}|\Delta_s|^4$, set $A_s:=\|\xi_s\|^2\ge0$ and $\mu_s:=\mathbb{E}[A_s\mid\mathcal{F}_{s-1}]$. The inequality $|a-b|^4\le 8(a^4+b^4)$ gives $|A_s-\mu_s|^4\le 8(A_s^4+\mu_s^4)$, and taking conditional expectations yields
\[
    \mathbb{E}\!\left[|A_s-\mu_s|^4\,\middle|\,\mathcal{F}_{s-1}\right]
    \le
    8\,\mathbb{E}[A_s^4\mid\mathcal{F}_{s-1}] + 8\,\mu_s^4.
\]
By conditional Jensen's inequality, $\mu_s^4=(\mathbb{E}[A_s\mid\mathcal{F}_{s-1}])^4
\le \mathbb{E}[A_s^4\mid\mathcal{F}_{s-1}]$, and hence
\[
    \mathbb{E}\!\left[|A_s-\mu_s|^4\,\middle|\,\mathcal{F}_{s-1}\right]
    \le
    16\,\mathbb{E}[A_s^4\mid\mathcal{F}_{s-1}]
    =
    16\,\mathbb{E}\!\left[\|\xi_s\|^{8}\mid\mathcal{F}_{s-1}\right].
\]
Consequently,
\[
    \mathbb{E}\!\left[|\Delta_s|^4\mid\mathcal{F}_{s-1}\right]
    =
    \eta_s^{8}\,
    \mathbb{E}\!\left[|A_s-\mu_s|^4\mid\mathcal{F}_{s-1}\right]
    \le
    16\,\eta_s^{8}\,\mathbb{E}\!\left[\|\xi_s\|^{8}\mid\mathcal{F}_{s-1}\right].
\]
Under Assumption~\ref{as:gradients} (conditional sub-Gaussian noise), the conditional $8$th moment is uniformly bounded; in particular, there exists $C_2<\infty$ such that $\mathbb{E}[\|\xi_s\|^{8}\mid\mathcal{F}_{s-1}]\le C_2$ almost surely for all $s$ (e.g., Lemma~\ref{lem:subg-moments}). Taking expectations gives
\[
    \mathbb{E}|\Delta_s|^4 \le 16C_2\,\eta_s^8,
    \qquad\text{so}\qquad
    (\mathbb{E}|\Delta_s|^4)^{1/2}\le (16C_2)^{1/2}\eta_s^4.
\]
Plugging this bound into \eqref{eq:minkowski-L2} yields
\[
    \mathbb{E}\!\left[\Big(\sum_{s=1}^T \Delta_s^2\Big)^2\right]
    \le
    (16C_2)\Big(\sum_{s=1}^T \eta_s^4\Big)^2
    \le
    (16C_2)\Big(\sum_{s=1}^\infty \eta_s^4\Big)^2
    <\infty.
\]
Combining with \eqref{eq:max4-via-square} shows that there exists a deterministic constant $C_3<\infty$ such that
\[
    \sup_{T\ge1}\mathbb{E}[\sup_{1\le t\le T}|M_t|^4]\le C_3.
\]
Since $\sup_{1\le t\le T}|M_t|^4$ increases to $\sup_{t\ge1}|M_t|^4$ as $T\to\infty$, monotone convergence yields
\[
    \mathbb{E}\!\left[\sup_{t\ge1}|M_t|^4\right] \le C_3 < \infty,
\]
which concludes the proof.
\end{proof}

\begin{lemma}
\label{lem:explicit-threshold}
Fix any $L_\alpha \ge 1$ and define $v_0 := 2(L_\alpha + 1)$. Then, for all $v \ge v_0$,
\[
    v > L_\alpha + \log\Big(\log(\mathrm{e}+v)\Big).
\]
\end{lemma}
\begin{proof}
Since $L_\alpha \ge 1$, we have $v_0 = 2(L_\alpha+1) \ge 4$. For any $v\ge 4$, we have $\mathrm{e}^v \ge \mathrm{e}+v$, which implies that $\log(\mathrm{e}+v) \le v$. Since $\log(\mathrm{e}+v) >1$, and therefore $\log\Big(\log(\mathrm{e}+v)\Big) \le \log v$ for all $v \ge 4$. Thus, for $v \ge v_0$, we have
\[
    v - L_\alpha - \log\Big(\log(\mathrm{e}+v)\Big)
    \ge
    v - L_\alpha - \log v
    =: g(v)
\]
Since $g'(v) = 1 - \tfrac{1}{v} \ge 0$ for all $v \ge 1$, we have that $g$ is increasing and
\[
    g(v) \ge g(v_0) = L_\alpha+2 - \log\Big(2(L_\alpha+1)\Big).
\]
Since $\log(2x) \le x$ for all $x \ge 1$, we obtain
\[
    g(v_0)
    \ge
    L_\alpha+2 - (L_\alpha+1)
    = 1
    > 0.
\]
Therefore $v - L_\alpha - \log(\log(\mathrm{e}+v)) > 0$ for all $v \ge v_0$, which concludes the proof.
\end{proof}

\begin{lemma}
\label{lem:poly-sum}
Let $\gamma\in(0,1)$ and define
\[
    S_t^{(\gamma)} := \sum_{s=1}^t s^{-\gamma}
    \qquad \forall t\ge 1.
\]
Then, for all $t\ge 1$,
\[
    S_t^{(\gamma)}
    \;\ge\;
    \frac{t^{1-\gamma}-1}{1-\gamma}.
\]
\end{lemma}
\begin{proof}
The function $x\mapsto x^{-\gamma}$ is positive and strictly decreasing on $[1,\infty)$, for any $\gamma>0$. Therefore, for each integer $t\ge1$,
\[
    \sum_{s=1}^t s^{-\gamma}
    \;\ge\;
    \int_{1}^{t+1} x^{-\gamma}\,dx.
\]
Evaluating the integral for $\gamma\in(0,1)$ gives
\[
    \int_{1}^{t+1} x^{-\gamma}\,dx
    \;=\;
    \left[\frac{x^{1-\gamma}}{1-\gamma}\right]_{x=1}^{x=t+1}
    \;=\;
    \frac{(t+1)^{1-\gamma}-1}{1-\gamma}.
\]
To conclude, note that $(t+1)^{1-\gamma}\ge t^{1-\gamma}$ for all $t\ge 1$ and $\gamma\in(0,1)$.
\end{proof}

\begin{lemma}
\label{lem:harmonic-sum}
Define
\[
    H_t := \sum_{s=1}^t \frac{1}{s},
    \qquad \forall t\ge 1.
\]
Then for all $t\ge 1$,
\[
    H_t
    \;\ge\;
    \log(t).
\]
\end{lemma}
\begin{proof}
The function $x\mapsto 1/x$ is positive and strictly decreasing on $[1,\infty)$. Hence, for each integer $t\ge1$,
\[
    \sum_{s=1}^t \frac{1}{s}
    \;\ge\;
    \int_{1}^{t+1} \frac{dx}{x}
    \;=\;
    \log(t+1),
\]
The final statement follows from the monotonicity of the logarithm: $\log(t+1)\ge \log t$ for all $t\ge 1$.
\end{proof}

\begin{lemma}
\label{lem:invert-b}
Let $b(t):=(1+\log t)/\sqrt t$ for $t\ge 1$, and define
\[
    T(u):=\inf\{t\ge 1:\ b(t)\le u\},
    \qquad u\ge 0,
\]
with the convention $\inf\emptyset=\infty$. Then there exists a universal constant $K<\infty$ such that for all $u\in(0,1/2)$,
\begin{equation}\label{eq:T-inv}
    T(u)\le K\,u^{-2}\,\log^2\!\frac{1}{u}.
\end{equation}
\end{lemma}
\begin{proof}
Fix $u\in(0,1/2)$ and set $L:=\log(1/u)$, so $L\ge \log 2$. Let $c\ge 1$ be a numerical
constant to be chosen and define $t_0:=c\,u^{-2}L^2$ and $t:=\lceil t_0\rceil$.
Since $u^{-2}\ge 4$ and $L\ge \log 2$, we have $t_0\ge 4c(\log 2)^2\ge 1$, hence
$t\le t_0+1\le 2t_0$ and therefore
\[
    \log t \le \log(2t_0)=\log(2c)+2\log\!\frac{1}{u}+2\log L
    =\log(2c)+2L+2\log L.
\]
Because $L\ge \log 2>0$ we have $\log L\le L$, which yields
\[
    1+\log t \le 1+\log(2c)+4L.
\]
Moreover, $\sqrt t\ge \sqrt{t_0}=\sqrt c\,u^{-1}L$, hence
\[
    b(t)=\frac{1+\log t}{\sqrt t}
    \le
    \frac{1+\log(2c)+4L}{\sqrt c\,u^{-1}L}
    =
    u\left(\frac{1+\log(2c)}{\sqrt c\,L}+\frac{4}{\sqrt c}\right)
    \le
    u\left(\frac{1+\log(2c)}{\sqrt c\,\log 2}+\frac{4}{\sqrt c}\right),
\]
where we used $L\ge \log 2$ in the last step. Choose $c$ large enough so that
$\frac{1+\log(2c)}{\sqrt c\,\log 2}+\frac{4}{\sqrt c}\le 1$
(which is possible since $(\log c)/\sqrt c\to 0$). Then $b(t)\le u$, and by the
definition of $T(u)$ we have $T(u)\le t\le 2t_0$. Therefore,
\[
    T(u)\le 2c\,u^{-2}L^2
    = 2c\,u^{-2}\log^2\!\frac{1}{u}.
\]
Setting $K:=2c$ proves \eqref{eq:T-inv}.
\end{proof}

\end{document}